\documentclass[11pt]{article}%
\usepackage{graphics,mathrsfs}
\usepackage{amsmath}
\usepackage{booktabs}
\usepackage{mathtools}
\usepackage{enumitem}
\usepackage{epsfig}
\usepackage{amssymb}
\usepackage{amsmath,bm}
\usepackage{caption}
\usepackage{bbold}
\usepackage{comment}

\usepackage{graphicx}
\usepackage{tikz}

\usepackage{geometry}
\newtheorem{theorem}{\bf Theorem}[section]
\newtheorem{proposition}{\bf Proposition}[section]

\newtheorem{corollary}{\bf Corollary}[section]
\newtheorem{remark}{\bf Remark}[section]
\newtheorem{definition}{\bf Definition}[section]

\newtheorem{example}{\bf Example }[section]
\numberwithin{equation}{section}

\usepackage{color}

\newenvironment{proof}{
	\begin{trivlist}
		\item[\hspace{\labelsep}{\em\noindent Proof.}]}{\hfill $\Box$\end{trivlist}}

\date{\plain}

\title{
	\huge\bf New multivariate Gini's indices  
	\date{December 14, 2023}
}
\author{
\large   
Marco {\bf Capaldo}$^1$\thanks{ORCID: 0000-0002-0255-8935} \qquad
Jorge {\bf Navarro} $^2$ \thanks{ORCID: 0000-0003-2822-915X}\\
\small $^1$ Dipartimento di Matematica, Universit\`a degli Studi di Salerno \\
\small Via Giovanni Paolo II, 132, I-84084 Fisciano (SA), Italy \\
\small $^2$ Departamento de Estadística e Investigación Operativa, Facultad de Matemática \\
\small Universidad de Murcia, Murcia 30100, Spain \\
\normalsize Email: mcapaldo@unisa.it, jorgenav@um.es  
}

\begin{document}
	\maketitle
	\abstract{\noindent The Gini's mean difference was defined as the expected absolute difference between a random variable and its independent copy. The corresponding normalized version, namely Gini's index, denotes two times the area between the egalitarian line and the Lorenz curve. Both are dispersion indices because they quantify how far a random variable and its independent copy are. Aiming to measure dispersion in the multivariate case, we define and study new Gini's indices. For the bivariate case we provide several results and we point out that they are ``dependence-dispersion" indices. Covariance representations are exhibited, with an interpretation also in terms of conditional distributions. Further results, bounds and illustrative examples are discussed too. Multivariate extensions are defined, aiming to apply both indices in more general settings. Then, we define efficiency Gini's indices for any semi-coherent system and we discuss about their interpretation. Empirical versions are considered in order as well to apply multivariate Gini's indices to data.
		
	}

	
	\section{Introduction}
	Information measures provide more specific details about random variables than those given by mean values, variances, and others moments. In this sense, they allow to broaden the knowledge about phenomena shaped by random variables. The Gini's index and the Gini's mean difference are two popular dispersion measures. The Gini's index is a well known measure of income inequality in a population related with the Lorenz curve. For its history and alternative formulations we refer the reader to Arnold and Sarabia \cite{Arnold:Sarabia} and Ceriani and Verme \cite{Ceriani:Verme}.
	\par 
	Given a nonnegative random variable $X$, with cumulative distribution function (c.d.f.) $F(t)=\mathsf{Pr}(X\le t)$ and survival function (s.f.) $\bar{F}= 1-F$, the Gini's mean difference of $X$ was defined as 
	\begin{equation}\label{eq:GMD}
		\mathsf{GMD}(X)\coloneqq\mathsf{E}|X-X'|=2\int_{0}^{+\infty}F(t)\bar{F}(t){\rm d}t,
	\end{equation}
	where $X'$ is an independent copy of $X$ (see, for instance, Arnold and Sarabia \cite{Arnold:Sarabia} and Yitzhaki \cite{Yitzhaki:2003}). This information measure quantifies the variability between two independent and identically distributed (i.i.d.) random variables by using the absolute mean difference. In other words, the Gini's mean difference measures how far they are respect to the egalitarian line, and in this sense it can be seen as dispersion index. A normalized version of Eq.\ (\ref{eq:GMD}), namely Gini's index of $X$, was defined as 
	\begin{equation}\label{eq:Gindex}
		\mathsf{G}(X)\coloneqq\frac{\mathsf{GMD}(X)}{2\mathsf{E}(X)}\in[0,1]
	\end{equation}
	where $0<\mathsf{E}(X)<\infty$. It represents twice the area between the egalitarian line and the Lorenz curve (see, for example, Gastwirth \cite{Gastwirth:1972} and Arnold and Sarabia \cite{Arnold:Sarabia}, p.\ 52). We point out that the Gini's indices given in Eqs.\ (\ref{eq:GMD}) and (\ref{eq:Gindex}) share many properties with the variance, since 
	\begin{equation}\label{eq:covariance Gini's} 
		\mathsf{GMD}(X)=4\mathsf{Cov}(X,F(X)),\qquad\mathsf{G}(X)=\frac{2}{\mathsf{E}(X)}\mathsf{Cov}(X,F(X)).
	\end{equation}
	They can be informative about the properties of distributions that depart from normality, see Yitzhaki \cite{Yitzhaki:2003}. For this and other reasons, Eq.\ (\ref{eq:Gindex}) plays an important role in economics and social sciences, when the normal distribution does not provide a good approximation to data (see Yitzhaki \cite{Yitzhaki:2003} and references therein). 
	\par
	Some generalizations and extensions of the Gini's index have been considered in the literature, see, for instance, Koshevoy and Mosler \cite{Koshevoy:Mosler}, Sections 5.3.2 and 7.4.1 in Arnold and Sarabia \cite{Arnold:Sarabia} and references therein. Moreover, different versions of the Gini's indices introduced above can be found in the context of concordance measures. For example, if $(X,Y)$ has copula $C$, with absolutely continuous marginal distributions, then the population version of Gini’s measure of association for $X$ and $Y$ is given by 
	$$
	\gamma_C=4\left(\int_0^1 C(u,1-u){\rm d}u-\int_0^1 (u-C(u,u)){\rm d}u\right),
	$$ 	
	see Section 5.1.4 of Nelsen \cite{Nelsen:2006} and Fuchs and Tschimpke \cite{Fuchs:Tschimpke}. Recently, the copula-distorted Gini's mean difference of $X$ has been defined in Capaldo et al.\ \cite{Capaldo:etal2}, aiming to generalize Eq.\ (\ref{eq:GMD}) for the cases in which the random variables involved are possibly dependent and where one s.f.\ (or c.d.f.) is distorted from the other, making use of a distortion function in a parametric family. Since the absolute mean difference can be view as a sort of distance between the random variables involved, related minimizing problems of distance can be found in Capaldo et al.\ \cite{Capaldo:etal2} and Ortega-Jim\'enez et al.\ \cite{Ortega-Jimenez:etal1}. Moreover, see Ortega-Jim\'enez et al.\ \cite{Ortega-Jimenez:etal2} for stochastic comparisons between absolute mean differences. 
	\par
	In this paper for a random vector $(X,Y)$ we propose and study a bivariate version of the Gini's mean difference in Eq.\ (\ref{eq:GMD}), defined as $\mathsf{GMD}(X,Y)=\mathsf{E}|X-Y|$. In this case the distribution of $Y$ is not necessarily equal to the distribution of $X$ and they can be dependent. We also define a bivariate version of the Gini's index in Eq.\ (\ref{eq:Gindex}).
	We extend both Gini's indices to the multivariate case. Moreover we focus on the meaning of the new measures that represent ``dependence-dispersion" indices. Along this line, for fixed marginal distributions, we prove that the indices decrease when the positive dependence increases. It is also related with the Schur-concavity of the underline copula. Additionally, for fixed dependence we prove that the indices decrease as the marginal dispersion decreases. We also show how they can be informative about the reliability of a system and, in this sense, we define their efficiency versions for any semi-coherent system. Empirical versions are also considered aiming to provide applications for the new Gini's indices to data. We show some illustrative examples where we find interpretations in terms of areas.  
	\par
	The paper is organized as follows. In Section \ref{sec:background} we recall useful notions regarding stochastic orders, copula functions with related properties, and some probabilistic inequalities. Section \ref{sec:bivariateGMD} is devoted to the bivariate Gini's mean difference and bivariate Gini's index. After definitions we obtain covariance representations for the new measures. Moreover, we show basic properties and provide several results, by extending them to the multivariate case. Section \ref{sec:bounds and inequalities} contains bounds and inequalities for both bivariate Gini's indices. In Section \ref{sec:effGMD} we define the efficiency Gini's mean difference for any semi-coherent system. We also define an efficiency version for the bivariate Gini's index and we discuss about its interpretation. Examples and simulations are given in Section \ref{sec:examples and simulations}, by providing also empirical versions for the new Gini's indices.

	\section{Background and notations}\label{sec:background}
	In this section we fix the notation by recalling preliminary notions such as usual stochastic and dispersive orders, copula functions and useful probabilistic inequalities.
	Throughout the paper the terms increasing and decreasing are used as ``non-decreasing" and ``non-increasing", respectively. In this paper we consider only nonnegative random variables, namely {\em random lifetime}. However we must note that some results also holds without this assumption. 
	\par
	For a random lifetime $X$ with c.d.f.\ $F$, the quantile function is denoted as $F^{-1}(u)=\sup\{x\colon F(x)\leq u\}$, for $u\in [0,1]$. In addition, in the absolutely continuous case, the probability density function (p.d.f.) is defined as $f(x)=F'(x)$ at points $x$ such that $F$ is derivable. We assume that $0<\mu=\mathsf{E}(X)=\int_{0}^{+\infty}\bar{F}(t){\rm d}t<\infty$, where $\bar{F}=1-F$.
	\par
	We say that $X$ is smaller than $Y$ in the usual {\em stochastic order}, denoted by $X\le_{st}Y$, if and only if $\bar{F}_X(t)\le\bar{F}_Y(t)$, for all $t\ge0$. If there is equality in law, then we write $X=_{st}Y$.
	We recall that $X$ is said to be smaller than $Y$ in the {\em dispersive order}, denoted as $X\le_d Y$, if and only if 
	$$
	F_X^{-1}(v)-F_X^{-1}(u) \ge F_Y^{-1}(v)- F_Y^{-1}(u)\qquad\hbox{whenever }0<u \leq v<1. 
	$$
	In addition, if $X$ and $Y$ are absolutely continuous, then  
	$$
	X\le_d Y \quad \hbox{if and only if} \quad f_X(F_X^{-1}(u))\ge f_Y(F_Y^{-1}(u))\quad\forall \, u\in(0,1).
	$$
	While the usual stochastic order is used to compare two random variables in size, the dispersive order compares their dispersion. Moreover, if $X\le_{st}Y$ then $\mathsf{E}(X)\le\mathsf{E}(Y)$, while for $X\le_{d}Y$ one has $\mathsf{Var}(X)\le\mathsf{Var}(Y)$. For more details about stochastic orders we refer the reader to Shaked and Shanthikumar \cite{Shaked:Shanthikumar}.
	\par
	From Sklar's theorem a multivariate distribution can be written in terms of the marginal distributions, making use of a suitable function, namely {\em copula function}. Thus, given a random vector ${\bf X}=(X_1,\dots,X_n)$ 
	with joint c.d.f.\
	$$
	F_{\bf X}(x_1,\dots,x_n)=\mathsf{Pr}(X_1\le x_1,\dots, X_n\le x_n)
	$$
	and marginal c.d.f.'s $F_{X_1},\dots,F_{X_n}$, it follows
	\begin{equation}\label{eq:copula}
		F_{\bf X}(x_1,\dots,x_n)=C(F_{X_1}(x_1),\dots,F_{X_n}(x_n)), 
	\end{equation}
	for $x_1,\dots,x_n\in\mathbb{R}$, where $C$ is called copula function of ${\bf X}$. If the marginal distributions are all continuous then the copula is unique. Similarly, referring to the joint s.f.\ 
	$$
	\bar{F}_{\bf X}(x_1,\dots,x_n)=\mathsf{Pr}(X_1> x_1,\dots, X_n> x_n)
	$$
	and marginal s.f.'s $\bar{F}_{X_1},\dots,\bar{F}_{X_n}$, it follows 
	\begin{equation}\label{eq:survival copula}
		\bar{F}_{\bf X}(x_1,\dots,x_n)=\widehat{C}(\bar{F}_{X_1}(x_1),\dots,\bar{F}_{X_n}(x_n)),
	\end{equation}
	for $x_1,\dots,x_n\in\mathbb{R}$, where $\widehat{C}$ is called {\em survival copula function} of ${\bf X}$. Moreover, the functions
	$$
	\delta(u)=C(u,\dots,u),\qquad\widehat{\delta}(u)=\widehat{C}(u,\dots,u),\qquad u\in[0,1],
	$$
	represent the diagonal sections of $C$ and $\widehat{C}$, respectively.
	\par 
	We recall that ${\bf X}=(X_1,\dots,X_n)$ is said to be {\em exchangeable} (exc.) if 
	$$
	(X_1,\dots,X_n)=_{st}(X_{\sigma(1)},\dots,X_{\sigma(n)})
	$$
	for all permutations $\sigma$ of order $n$. From exchangeability it follows that $X_1,\dots,X_n$ are identically distributed (i.d.).
	\par 
	For simplicity in the bivariate case we now recall some properties regarding copulas. First, for all $u,v\in[0,1]$ the relation between $C$ and $\widehat{C}$ is given by
	\begin{equation}\label{eq:copula and survival copula}
		\widehat{C}(u,v)=u+v-1+C(1-u,1-v).
	\end{equation}
	Moreover, for every copula $C$ and for all $u,v\in[0,1]$ one has
	\begin{equation}\label{eq:Fréchet-Hoeffding bounds}
		W(u,v)\le C(u,v)\le M\{u,v\},
	\end{equation}
	where $W(u,v)=\max\{u+v-1,0\}$ and $M(u,v)=\min\{u,v\}$ are two copulas called {\em Fréchet-Hoeffding (FH) bounds} (lower and upper, respectively). 
	Another important copula that we will encounter is the independence copula, usually denoted as $\Pi(u,v)=uv$, for all $u,v\in[0,1]$. Moreover, we will also consider the Farlie-Gumbel-Morgenstern (FGM) family of copulas, defined for all $u,v\in[0,1]$ as
	\begin{equation}\label{eq:FGM}
		C_{\theta}(u,v)=uv\left(1+\theta(1-u)(1-v)\right),\qquad\theta\in[-1,1].
	\end{equation}
	Clearly, from Eq.\ (\ref{eq:FGM}) one has $C_{0}(u,v)=\Pi(u,v)$. A family of copulas that includes $M,\Pi$ and $W$ is called {\em comprehensive}. An example is the Clayton family of copulas, defined for all $u,v\in[0,1]$ as
	\begin{equation}\label{eq:Clayton}
		C_{\theta}(u,v)=\left[\max(u^{-\theta}+v^{-\theta}-1,0)\right]^{-\frac{1}{\theta}},\qquad\theta\in[-1,0)\cup(0,+\infty),
	\end{equation}
	where $C_{-1}=W$, $C_0=\Pi$ and $C_{+\infty}=M$. Another example of comprehensive family of copulas is represented by the Frank family of copulas, defined for all $u,v\in[0,1]$ as
	\begin{equation}\label{eq:Frank}
		C_{\theta}(u,v)=-\frac{1}{\theta}\ln\left(1-\frac{(1-e^{-\theta u})(1-e^{-\theta v})}{e^{-\theta }-1}\right),\qquad\theta\in(-\infty,0)\cup(0,+\infty),
	\end{equation}
	where $C_{-\infty}=W$, $C_0=\Pi$ and $C_{+\infty}=M$.
	\par
	For our aims, we recall that a copula $C_1$ is smaller than another copula $C_2$ under the {\em concordance order}, and write $C_1\prec C_2$, if $C_1(u,v)\le C_2(u,v)$ for all $u,v\in[0,1]$. Moreover, a random vector $(X,Y)$ with a copula $C$ is said to be positively quadrant dependent (PQD), if and only if $\Pi\prec C$. Conversely it has the negatively quadrant dependent (NQD) property if and only if $C\prec\Pi$. Throughout the paper, we will denote with $\partial_1{C}$ (respectively $\partial_2{C}$) the partial derivative with respect to the first (second) argument of $C$. For these and other basic properties of copulas we refer the reader to Durante and Sempi \cite{Durante:Sempi} and Nelsen \cite{Nelsen:2006}.
	\par 
	In addition, we now recall some properties for bivariate c.d.f.'s or s.f.'s, that clearly are related with copulas and survival copulas from Eqs.\ (\ref{eq:copula}) and (\ref{eq:survival copula}). Let $(X,Y)$ be a random vector with joint s.f.\ $\bar{H}$. We say that $\bar{H}$ is Schur-concave (Schur-convex) if
	\begin{equation}\label{eq:Schur}
		\bar{H}(x,y)\leq\bar{H}(\tilde x,\tilde y)\,(\geq),
	\end{equation}
	for all $x+y=\tilde x+\tilde y$ and $\min(x,y)\leq \min (\tilde x,\tilde y)$, see Durante and Sempi \cite{Durante:Sempi}, p.\ 278. We remark that the Fréchet-Hoeffding lower bound in Eq.\ (\ref{eq:Fréchet-Hoeffding bounds}) represents the only Schur-convex copula (see, e.g., Nelsen \cite{Nelsen:2006}, p.\ 104). This properties can be relaxed as follows: $\bar{H}$ is weakly Schur-concave (Schur-convex) if
	\begin{equation}\label{eq:weakly Schur}
		\bar{H}(x,y)\leq\bar{H}(z,z)\,(\geq),
	\end{equation}
	for all $x,y$ and $z=(x+y)/2$, see Durante and Papini \cite{Durante:Papini}. Moreover, $\bar{H}$ is Schur-constant if 
	\begin{equation}\label{eq:Schur-constant}
		\bar{H}(x,y)=\bar{G}(x+y), 
	\end{equation}
	for all $x,y\ge0$, where $\bar{G}$ is a univariate s.f., see Caramellino and Spizzichino \cite{Caramellino:Spizzichino}.
	We remark that in the case of Schur-constant $X$ and $Y$ are exc. See also Pellerey and Navarro \cite{Pellerey:Navarro} for Schur-constant joint s.f.\ related to monotonicity properties of dependent variables given their sum.
	\par 
	We conclude this section by recalling well known probabilistic inequalities, useful for our aims. 
	Given a random variable $Z$ and a convex function $\psi$, the Jensen's inequality states that 
	\begin{equation}\label{eq:Jensen's inequality}
		\psi\left(\mathsf{E}(Z)\right)\le\mathsf{E}\left(\psi(Z)\right).
	\end{equation}
	Moreover, if $Z$ is a random lifetime and $a>0$, the Markov's inequality states that
	\begin{equation}\label{eq:Markov's inequality}
		\mathsf{Pr}\left(Z\ge a\right)\le\frac{\mathsf{E}(Z)}{a}.
	\end{equation}
	%

	\section{Bivariate Gini's indices}\label{sec:bivariateGMD}
	We now define a bivariate version of the Gini's mean difference introduced in Eq.\ (\ref{eq:GMD}), and we study its basic properties. The formal definition can be stated as follows. 
	\begin{definition}\label{def:bivariateGMD}
		Let	$(X,Y)$ be a random vector. The bivariate Gini's mean difference of $(X,Y)$ is defined as 
		\begin{equation}\label{eq:bivariateGMD}
			\mathsf{GMD}(X,Y)=\mathsf{E}|X-Y|.
		\end{equation}
	\end{definition}
	\par
	For our aims, we denote $U=\max\{X,Y\}$ and $L=\min\{X,Y\}$. By using $|x-y|=\max\{x,y\}-\min\{x,y\}$ in Eq.\ (\ref{eq:bivariateGMD}), it follows 
	\begin{equation}\label{eq:mean difference}
		\mathsf{GMD}(X,Y)=\mathsf{E}(U)-\mathsf{E}(L).
	\end{equation}
	Therefore, Eq.\ (\ref{eq:mean difference}) can be seen as the expected distance between the lifetimes of a parallel system and a series system, both having two components distributed as $X$ and $Y$, respectively. Note that $X$ and $Y$ can be dependent.  Moreover, since $|x-y|=x+y-2\min\{x,y\}$, from Eq.\ (\ref{eq:bivariateGMD}) it also follows 
	\begin{equation}\label{eq:bivariateGMD mean}
		\mathsf{GMD}(X,Y)=\mathsf{E}(X)+\mathsf{E}(Y)-2\mathsf{E}(L), 
	\end{equation}
	and, by using Eqs.\ (\ref{eq:mean difference}) and (\ref{eq:bivariateGMD mean}), one has the well known property  $\mathsf{E}(X)+\mathsf{E}(Y)=\mathsf{E}(U)+\mathsf{E}(L)$.
	\par 
	Suppose now that $(X,Y)$ has survival copula $\widehat{C}$, with marginal s.f.'s $\bar{F}_X,\bar{F}_Y$, respectively. Then, for all $t\ge 0$ one has 
	$$
	\bar{F}_{L}(t)=\mathsf{Pr}\left(X>t,Y>t\right)=\widehat{C}(\bar{F}_X(t),\bar{F}_Y(t))
	$$
	and therefore Eq.\ (\ref{eq:bivariateGMD mean}) can be rewritten as
	\begin{equation}\label{eq:bivariateGMD s.f.}
		\mathsf{GMD}(X,Y)=\int_{0}^{+\infty}\left\{\bar{F}_X(t)+\bar{F}_Y(t)-2\widehat{C}(\bar{F}_X(t),\bar{F}_Y(t))\right\}{\rm d}t.
	\end{equation}
	Clearly, if $X$ and $Y$ are i.i.d.\ Eq.\ (\ref{eq:bivariateGMD s.f.}) is the univariate Gini's mean difference in Eq.\ (\ref{eq:GMD}). If $\bar{F}_Y(t)=h_\alpha(\bar{F}_X(t))$ for all $t$, where $h_\alpha$ is a distortion function in a parametric family with  real value $\alpha$ and, by considering also a parametric family of survival copulas for $(X,Y)$, Eq.\ (\ref{eq:bivariateGMD s.f.}) reduces to the copula-distorted Gini's mean difference defined in Capaldo et al.\ \cite{Capaldo:etal2}. One can use the c.d.f.'s configuration instead of the s.f.'s one, since from Eqs.\ (\ref{eq:copula and survival copula}) and (\ref{eq:bivariateGMD s.f.}) it follows
	\begin{equation}\label{eq:bivariateGMD c.d.f.}
		\mathsf{GMD}(X,Y)=\int_{0}^{+\infty}\left\{F_X(t)+F_Y(t)-2C(F_X(t),F_Y(t))\right\}{\rm d}t,
	\end{equation}
	where $C$ denotes the copula function of $(X,Y)$, and $F_X$, $F_Y$ the marginal c.d.f.'s.
	\par
	As for the univariate case in Eq.\ (\ref{eq:Gindex}), we now define the bivariate Gini's index as a ratio related to Eq.\ (\ref{eq:bivariateGMD}). 
	\begin{definition}
		Let $(X,Y)$ be a random vector. The bivariate Gini's index of $(X,Y)$ is defined as 
		\begin{equation}\label{eq:bivariateGindex}
			\mathsf{G}(X,Y)=\frac{\mathsf{GMD}(X,Y)}{\mathsf{E}(X)+\mathsf{E}(Y)}.
		\end{equation}
	\end{definition}
	Clearly, if $X$ and $Y$ are i.i.d., Eq.\ (\ref{eq:bivariateGindex}) is the univariate Gini's index in Eq.\ (\ref{eq:Gindex}). We remark that Eq.\ (\ref{eq:bivariateGindex}) extends the copula-distorted Gini's index defined in \cite{Capaldo:etal2}. Recalling Eq.\ (\ref{eq:mean difference}), one has  
	\begin{equation}\label{eq:mean index}
		\mathsf{G}(X,Y)=\frac{\mathsf{E}(U)-\mathsf{E}(L)}{\mathsf{E}(U)+\mathsf{E}(L)}.
	\end{equation}
	Hence $\mathsf{G}(X,Y)\in[0,1]$, with $\mathsf{G}(X,Y)=0$ if and only if $\mathsf{E}(U)=\mathsf{E}(L)$.
	In addition, $\mathsf{G}(X,Y)=1$ when $\mathsf{E}(L)=0$.
	\par
	As for the univariate case (see Eq.\ (\ref{eq:covariance Gini's})), we now explore the possibility to express the bivariate Gini's indices in terms of a covariance between the random lifetimes involved and their suitable transformations based on distributions. For this reason, given a random vector $(X,Y)$ we define the following functions
	\begin{equation}\label{eq:gamma with conditional}
		\gamma_1(t)\coloneqq{\mathsf{Pr}\left(Y>t|X=t\right)},\qquad\gamma_2(t)\coloneqq{\mathsf{Pr}\left(X>t|Y=t\right)},\quad t\ge 0.	
	\end{equation}
	If $\widehat{C}$ denotes the survival copula of $(X,Y)$, then it is easy to get (see, for instance, Navarro and Sordo \cite{Navarro:Sordo} or Nelsen \cite{Nelsen:2006}, p.\ 217)
	\begin{equation}\label{eq:gamma with copula}
		\gamma_1(t)=\partial_1\widehat{C}\left(\bar{F}_X(t),\bar{F}_Y(t)\right),\qquad \gamma_2(t)=\partial_2\widehat{C}\left(\bar{F}_X(t),\bar{F}_Y(t)\right),\quad t\ge 0.	
	\end{equation}
	In the next result we first obtain a covariance representation of $\mathsf{E}(L)$. 
	\begin{proposition}
		Let $(X,Y)$ be a random vector with absolutely continuous distribution. Then 
		\begin{equation}\label{eq:covariance representation of E(L)}
			\mathsf{E}(L)=\mathsf{Cov}\left(X,\gamma_1(X)\right)+\mathsf{Cov}\left(Y,\gamma_2(Y)\right)+\mathsf{E}(X)\mathsf{E}(\gamma_1(X))+\mathsf{E}(Y)\mathsf{E}(\gamma_2(Y)).
		\end{equation}
	\end{proposition}
	\begin{proof}
		Since $\bar{F}_L(t)=\widehat{C}\left(\bar{F}_X(t),\bar{F}_Y(t)\right)$, the p.d.f.\ of $L$ is
		$$
		f_L(t)=f_X(t)\partial_1\widehat{C}\left(\bar{F}_X(t),\bar{F}_Y(t)\right)+f_Y(t)\partial_2\widehat{C}\left(\bar{F}_X(t),\bar{F}_Y(t)\right),\qquad t\ge0,
		$$
		where $f_X$ and $f_Y$ are the marginal p.d.f.'s. Therefore, recalling Eq.\ (\ref{eq:gamma with copula}), one has 
		$$
		\mathsf{E}(L)=\mathsf{E}(X\gamma_1(X))+\mathsf{E}(Y\gamma_2(Y)).
		$$
		Finally, Eq.\ (\ref{eq:covariance representation of E(L)}) follows from the definition of covariance.
	\end{proof}
	If $X$ and $Y$ are absolutely continuous, then from Eq.\ (\ref{eq:gamma with conditional}) one has $\mathsf{E}(\gamma_1(X))=\mathsf{Pr}\left(Y>X\right)$ and $\mathsf{E}(\gamma_2(Y))=1-\mathsf{E}(\gamma_1(X))$.
	Therefore Eq.\ (\ref{eq:covariance representation of E(L)}) can also be written as
	\begin{equation}\label{eq:covariance representation of E(L) 2}
		\mathsf{E}(L)=\mathsf{Cov}\left(X,\gamma_1(X)\right)+\mathsf{Cov}\left(Y,\gamma_2(Y)\right)+\left(\mathsf{E}(X)-\mathsf{E}(Y)\right)\mathsf{Pr}\left(Y>X\right)+\mathsf{E}(Y).
	\end{equation}
	Along the same line, in the next result we obtain a covariance representation for the bivariate Gini's mean difference.
	Recalling Eq.\ (\ref{eq:bivariateGindex}), dividing by $\mathsf{E}(X)+\mathsf{E}(Y)$ one can obtain a covariance representation for the bivariate Gini's index. 
	\begin{theorem}\label{th: covariance bivariate GMD}
		Let $(X,Y)$ be a random vector with absolutely continuous distribution. Then 
		\begin{equation*}\label{eq:cov bivariate GMD}
			\mathsf{GMD}(X,Y)=2\left[\mathsf{E}(X)-\mathsf{E}(Y)\right]\left[\frac12-\mathsf{Pr}(Y>X)\right]-2\mathsf{Cov}(X,\gamma_1(X))-2\mathsf{Cov}(Y,\gamma_2(Y)). 
		\end{equation*}
	\end{theorem}
	\begin{proof}
		The thesis immediately follows from Eqs.\ (\ref{eq:bivariateGMD mean}) and (\ref{eq:covariance representation of E(L) 2}).
	\end{proof}
	\begin{corollary}
		Under the same assumptions of Theorem \ref{th: covariance bivariate GMD}:
		\begin{description}
			\item[(i)] if $\mathsf{E}(X)=\mathsf{E}(Y)$ or if $\mathsf{Pr}(Y>X)=1/2$, then
			$$
			\mathsf{GMD}(X,Y)=-2\mathsf{Cov}(X,\gamma_1(X))-2\mathsf{Cov}(Y,\gamma_2(Y));
			$$
			\item[(ii)] if $X$ and $Y$ are exc., then 
			\begin{equation}\label{eq:covariance bivariate GMD exc.}
				\mathsf{GMD}(X,Y)=-4\mathsf{Cov}(X,\gamma_1(X)).
			\end{equation}
		\end{description} 
	\end{corollary}
	\par 
	From the left-hand-side of Eq.\ (\ref{eq:gamma with conditional}), we denote with $F_{2|1}(t|t)=\mathsf{Pr}\left(Y\le t|X=t\right)=1-\gamma_1(t)$, for all $t\ge0$, the probability that the second lifetime is less or equal to $t$ conditioned by the fact that the first lifetime is equal to $t$. If $X$ and $Y$ are exc., from Eqs.\ (\ref{eq:bivariateGindex}) and (\ref{eq:covariance bivariate GMD exc.}), it follows  
	\begin{equation}\label{eq:cov bivariate Gini index}
		\mathsf{G}(X,Y)=\frac{2}{\mathsf{E}(X)}\mathsf{Cov}(X,F_{2|1}(X|X))=-\frac{2}{\mathsf{E}(X)}\mathsf{Cov}(X,\gamma_1(X)),
	\end{equation}
	that extends the covariance representation of the univariate Gini's index in Eq.\ (\ref{eq:covariance Gini's}). Under the exchangeability assumption, Eq.\ (\ref{eq:cov bivariate Gini index}) represents another way to compute the bivariate Gini's index making use of conditional distributions instead of copula representations. 
	\par
	As an example, given a random vector $(X,Y)$, let us consider the exponential conditional distributions with joint p.d.f.\ given by 
	$$
	f(x,y)=\exp\{-(c_{11}+c_{12}x+c_{21}y+c_{22}xy)\},\qquad x>0,y>0,
	$$
	where $c_{12}>0$, $c_{21}>0$, $c_{22}\ge0$ (the case $c_{22}=0$ corresponds to the independent case) and 
	\begin{equation}\label{eq:c11}
		c_{11}=\ln\left[\frac{1}{c_{12}c_{21}}\int_{0}^{+\infty}e^{-u}\left(1+\frac{c_{22}u}{c_{12}c_{21}}\right)^{-1}{\rm d}u\right].	
	\end{equation}
	Therefore one has  
	$$
	\mathsf{Pr}\left(Y>y|X=x\right)=\exp\{-(c_{21}+c_{22}x)y\},\qquad x>0,y>0,
	$$
	(see Arnold et al.\ \cite{Arnold:etal1} and \cite{Arnold:etal2}).
	For instance, if $c_{12}=c_{21}=c_{22}=1$, from Eq.\ (\ref{eq:c11}) one has $c_{11}=-0.516932$. Hence, recalling Eq.\ (\ref{eq:gamma with conditional}), it follows
	$$
	\gamma_1(t)=\exp\{-(t+t^2)\},\qquad t>0. 
	$$
	Under these assumptions, $X$ and $Y$ are exc.\ and one numerically has $\mathsf{E}(X)=0.676875$, $\mathsf{E}(\gamma_1(X))=0.5$ and $\mathsf{E}(X\gamma_1(X))=0.135428$. Therefore, from the definition of covariance and Eq.\ (\ref{eq:cov bivariate Gini index}), one obtains $\mathsf{G}(X,Y)=0.599843$.
	\par 
	In the next result we show some basic properties satisfied by the bivariate Gini's mean difference.
	\begin{theorem}\label{th:properties GMD}
		The bivariate Gini's mean difference of $(X,Y)$ satisfies the following properties:
		\begin{description}
			\item{(i)} (law invariance) if $(X,Y)=_{st}(\tilde{X},\tilde{Y})$, then $\mathsf{GMD}(X,Y)=\mathsf{GMD}(\tilde{X},\tilde{Y})$;
			\item{(ii)} (translation invariance) $\mathsf{GMD}(X+\lambda,Y+\lambda)=\mathsf{GMD}(X,Y)$, for all $\lambda\in\mathbb{R}$;
			\item{(iii)} (homogeneity) $\mathsf{GMD}(\lambda X,\lambda Y)=|\lambda|\mathsf{GMD}(X,Y)$, for all $\lambda\in\mathbb{R}$;
			\item{(iv)} (non-negativity) $\mathsf{GMD}(X,Y)\ge0$ for all $X$ and $Y$. In addition
			if $X$ and $Y$ are i.d.\ and $C(u,v)=\min\{u,v\}$, for all $u,v\in[0,1]$, then $\mathsf{GMD}(X,Y)=0$.
			Conversely, if $\mathsf{GMD}(X,Y)=0$, then $F_X=F_Y$ and if $F_X$ is continuous then $C(u,v)=\min\{u,v\}$, for all $u,v\in[0,1]$.
		\end{description}
	\end{theorem}
	\begin{proof}
		The property (i) is trivial. The properties (ii) and (iii) immediately follows from Eq.\ (\ref{eq:bivariateGMD}). 
		$\mathsf{GMD}(X,Y)\ge0$ also holds trivially from Eq.\ (\ref{eq:bivariateGMD}). If $X$ and $Y$ are i.d.\ and $C(u,v)=\min\{u,v\}$, for all $u,v\in[0,1]$, then $\mathsf{GMD}(X,Y)=0$ from Eq.\ (\ref{eq:bivariateGMD c.d.f.}). 
		Conversely, if $\mathsf{GMD}(X,Y)=0$, from Eq.\ (\ref{eq:mean difference}) one has $\mathsf{E}(U)=\mathsf{E}(L)$ and since $L\le_{st}U$, from Theorem 1.A.8.\ in \cite{Shaked:Shanthikumar}, it follows $\bar{F}_L=\bar{F}_U$. Therefore, since $L\le_{st}X\le_{st}U$ and $L\le_{st}Y\le_{st}U$ then $F_X=F_Y$. Under these assumptions and if $F_X$ is continuous, from Eq.\ (\ref{eq:bivariateGMD c.d.f.}) one has $\delta(u)=u$ for all $u\in[0,1]$. Therefore for all $u,v\in[0,1]$ 
		$$
		C(u,v)\ge C(\min\{u,v\},\min\{u,v\})=\delta(\min\{u,v\})=\min\{u,v\},
		$$
		and, by recalling $M$ in Eq.\ (\ref{eq:Fréchet-Hoeffding bounds}), finally the thesis holds.
	\end{proof}
	Similarly, from Eq.\ (\ref{eq:bivariateGindex}), it is easy to see that the bivariate Gini's index satisfies the following properties:
	\begin{description}
		\item{(i)} (law invariance) if $(X,Y)=_{st}(\tilde{X},\tilde{Y})$, then $\mathsf{G}(X,Y)=\mathsf{G}(\tilde{X},\tilde{Y})$;
		\item{(ii)} (translation) $\mathsf{G}(X+\lambda,Y+\lambda)=\frac{\mathsf{E}(X)+\mathsf{E}(Y)}{\mathsf{E}(X)+\mathsf{E}(Y)+2\lambda}\mathsf{G}(X,Y)$, for all $\lambda\in\mathbb{R}$;
		\item{(iii)} (homogeneity invariance) $\mathsf{G}(\lambda X,\lambda Y)=\mathsf{G}(X,Y)$, for all $\lambda\in\mathbb{R}^+$;
		\item{(iv)} (non-negativity) $\mathsf{G}(X,Y)\ge0$ for all $X$ and $Y$. 
		In addition $\mathsf{G}(X,Y)=0$ under the same conditions given in (iv) of Theorem \ref{th:properties GMD} for $\mathsf{GMD}(X,Y)$.
	\end{description}
	\par
	Aiming to compare different bivariate Gini's mean difference, below we provide some results under suitable assumptions on the marginal distributions and copula functions involved in Eqs.\ (\ref{eq:bivariateGMD s.f.}) and (\ref{eq:bivariateGMD c.d.f.}).
	First we show a dependence ordering for the bivariate Gini's mean difference and bivariate Gini's index through concordance order between two copula functions when we maintain the marginals. 
	\begin{proposition}\label{prop:concordance ordering}
		Let $(X,Y)$ and $(\tilde{X},\tilde{Y})$ be two random vectors with copulas $C_1$ and $C_2$, respectively. Suppose that $X=_{st}\tilde{X}$ and $Y=_{st}\tilde{Y}$, with c.d.f.'s $F_X$ and $F_Y$, respectively. If 
		$C_1\prec C_2$, then 
		$$
		\mathsf{GMD}(X,Y)\ge\mathsf{GMD}(\tilde{X},\tilde{Y}),\qquad\mathsf{G}(X,Y)\ge\mathsf{G}(\tilde{X},\tilde{Y}).
		$$
	\end{proposition}
	\begin{proof}
		Under these assumptions, recalling Eq.\ (\ref{eq:bivariateGMD c.d.f.}), one has
		\begin{equation}\label{eq:proof prop concordance ordering}
			\mathsf{GMD}(X,Y)-\mathsf{GMD}(\tilde{X},\tilde{Y})=2\int_{0}^{+\infty}\left\{C_2\left(F_X(t),F_Y(t)\right)-C_1\left(F_X(t),F_Y(t)\right)\right\}{\rm d}t. 	
		\end{equation}
		Hence, recalling also Eq.\ (\ref{eq:bivariateGindex}), the thesis follows from $C_1\prec C_2$.  
	\end{proof}
	\par
	Therefore if we fix the marginals and the dependence increases in the concordance order, then $\mathsf{GMD}(X,Y)$ and $\mathsf{G}(X,Y)$ decrease, that is, these dispersions indices decrease when the positive dependence increases. When all the marginals are i.d.\ we have the following result.
	\begin{proposition}\label{prop:concordance ordering diagonal section}
		Let $(X,Y)$ and $(\tilde{X},\tilde{Y})$ be two random vectors with diagonal sections $\delta_1$ and $\delta_2$, respectively. Suppose that all the marginals are i.d.\ with common c.d.f.\ $F$.
		If $\delta_1\le\delta_2$, then 
		$$
		\mathsf{GMD}(X,Y)\ge\mathsf{GMD}(\tilde{X},\tilde{Y}),\qquad\mathsf{G}(X,Y)\ge\mathsf{G}(\tilde{X},\tilde{Y}),
		$$
		with equality reached when $\delta_1=\delta_2$.
	\end{proposition}
	\begin{proof}
		The thesis immediately follows from Eq.\ (\ref{eq:proof prop concordance ordering}) when $X=_{st}Y$. 
	\end{proof}
	Making use of Corollary 8 in Ortega-Jim\'enez et al.\ \cite{Ortega-Jimenez:etal2} and Proposition \ref{prop:concordance ordering diagonal section}, we can go further and get the following result. 
	\begin{theorem}
		Let $(X,\tilde{X})$ and $(Y,\tilde{Y})$ be two random vectors with diagonal sections $\delta_1$ and $\delta_2$, respectively. Suppose that $X=_{st}\tilde{X}$ and $Y=_{st}\tilde{Y}$. If $Y\le_dX$ and $\delta_1\le\delta_2$, then 
		$$
		\mathsf{GMD}(X,\tilde{X})\ge\mathsf{GMD}(Y,\tilde{Y}).
		$$
	\end{theorem}
	\begin{proof}
		Consider $(Y_1,\tilde{Y}_1)$ with diagonal section $\delta_1$ and suppose that $Y=_{st}Y_1=_{st}\tilde{Y}_1$. 
		The thesis follows since from Corollary 8 in Ortega-Jim\'enez et al.\ \cite{Ortega-Jimenez:etal2} and Proposition \ref{prop:concordance ordering diagonal section} one has, respectively,  
		$$
		\mathsf{GMD}(X,\tilde{X})\ge\mathsf{GMD}(Y_1,\tilde{Y}_1)\ge\mathsf{GMD}(Y,\tilde{Y}).
		$$
	\end{proof}
	\par
	In the following result we compare the bivariate Gini's mean difference in the cases of homogeneous and heterogeneous marginal distributions.
	\begin{proposition}\label{prop:GMD homogeneous and heterogeneous}
		Let $(X,Y)$ be a random vector with survival copula $\widehat C$ and strictly increasing diagonal section $\widehat \delta$. Let $(Z,\tilde{Z})$ be another random vector with the same survival copula and common marginal s.f.\ $\bar{G}$. If
		\begin{equation}\label{eq:GMD homogeneous and heterogeneous}
			\widehat \delta^{-1} (\widehat C(\bar{F}_X(t),\bar{F}_Y(t)))\leq \bar{G}(t) \leq \frac{\bar{F}_X(t)+\bar{F}_Y(t)} 2
		\end{equation}
		holds for all $t$, then $\mathsf{GMD}(X,Y)\geq\mathsf{GMD}(Z,\tilde{Z})$.
	\end{proposition}
	\begin{proof}
		From Eqs.\ (\ref{eq:bivariateGMD s.f.}) and (\ref{eq:GMD homogeneous and heterogeneous}), as $\widehat C(\bar{F}_X(t),\bar{F}_Y(t)))\leq \widehat \delta (\bar{G}(t))$, one has 
		\begin{align*}
			\mathsf{GMD}(X,Y)
			&=2\int_0^\infty \left( \frac{\bar{F}_X(t)+\bar{F}_Y(t)} 2-\widehat C(\bar{F}_X(t),\bar{F}_Y(t))  \right){\rm d}t \\
			&\geq 2\int_0^\infty \left( \bar{G}(t)-\widehat \delta (\bar{G}(t)) \right){\rm d}t \\
			&=\mathsf{GMD}(Z,\tilde{Z})
		\end{align*}
		and the proof is complete.
	\end{proof}
	\par
	Note that both $\bar{G}_1(t):=\widehat \delta^{-1} ( \widehat C(\bar{F}_X(t),\bar{F}_X(t)))$ and $\bar{G}_2(t):=(\bar{F}_X(t)+\bar{F}_Y(t))/ 2$ are s.f.'s obtained as distortions of $\bar{F}_X$ and $\bar{F}_Y$. We also remark that the bivariate Gini's mean difference decreases when the dependence does not change and the marginal distributions are homogeneous, that is $Z$ and $\tilde{Z}$ are more similar than the original $X$ and $Y$. For example, when $\widehat{C}=\Pi$, then we can apply Proposition \ref{prop:GMD homogeneous and heterogeneous} if $\bar{G}$ is between the geometric and the arithmetic means of $\bar{F}_X$ and $\bar{F}_Y$, since the condition expressed by Eq.\ (\ref{eq:GMD homogeneous and heterogeneous}) becomes 
	$$
	\bar{G}_1(t)=\sqrt{\bar{F}_X(t)\bar{F}_Y(t)}\leq \bar{G}(t)\leq \bar{G}_2(t)=\frac{\bar{F}_X(t)+\bar{F}_Y(t)} 2.
	$$
	Note that we can also apply Proposition \ref{prop:GMD homogeneous and heterogeneous} to both $\bar{G}_1$ and $\bar{G}_2$. A similar result can be stated for the c.d.f.\ configuration from Eq.\ (\ref{eq:bivariateGMD c.d.f.}), as follows.
	\begin{proposition}
		Let $(X,Y)$ be a random vector with copula $C$ and strictly increasing diagonal section $ \delta$. Let $(Z,\tilde{Z})$ be another random vector with the same copula and common marginal c.d.f.\ $G$. If
		$$ \delta^{-1} ( C(F_X(t), F_Y(t)))\leq  G(t) \leq \frac{F_X(t)+F_Y(t)} 2$$
		holds for all $t$, then $\mathsf{GMD}(X,Y)\geq\mathsf{GMD}(Z,\tilde{Z})$.
	\end{proposition}
	
	We can also obtain a result in order to compare heterogeneous cases. It can be stated as follows. A similar result holds for the c.d.f.\ configuration given in Eq.\ (\ref{eq:bivariateGMD c.d.f.}). 
	\begin{proposition}\label{prop:GMD heterogeneous}
		Let $(X,Y)$ be a random vector with survival copula $\widehat C$. Let $(\tilde X,\tilde Y)$ be another random vector with the same survival copula. If $\mathsf{E}(X)+\mathsf{E}(Y)\ge\mathsf{E}(\tilde X)+\mathsf{E}(\tilde Y)$ and 
		$$
		\widehat C(\bar{F}_X(t),\bar{F}_Y(t))\leq \widehat C(\bar{F}_{\tilde X}(t),\bar{F}_{\tilde Y}(t))
		$$ 
		holds for all $t$, then $\mathsf{GMD}(X,Y)\geq\mathsf{GMD}(\tilde X,\tilde Y)$.
	\end{proposition}
	We remark that, under the assumptions of Proposition \ref{prop:GMD heterogeneous}, when $\mathsf{E}(X)+\mathsf{E}(Y)=\mathsf{E}(\tilde X)+\mathsf{E}(\tilde Y)$ one also has $\mathsf{G}(X,Y)\geq\mathsf{G}(\tilde X,\tilde Y)$.
	Moreover, recalling Eq.\ (\ref{eq:Schur}), if the survival copula is Schur-concave we can go further and get the following result. 	
	\begin{proposition}\label{prop:Schur-concave}
		Let $(X,Y)$ be a random vector with survival copula $\widehat C$. Let $(\tilde X,\tilde Y)$ be another random vector with the same survival copula. If $\widehat C$ is Schur-concave, $\bar{F}_X+\bar{F}_Y=\bar{F}_{\tilde X}+\bar{F}_{\tilde Y}$ and $\min(\bar{F}_X,\bar{F}_Y)\leq\min(\bar{F}_{\tilde X},\bar{F}_{\tilde Y})$, then
		$$
		\mathsf{GMD}(X,Y)\geq\mathsf{GMD}(\tilde X,\tilde Y),\qquad\mathsf{G}(X,Y)\geq\mathsf{G}(\tilde X,\tilde Y).
		$$
	\end{proposition}
	\begin{proof}
		Under the stated assumptions, if $\widehat C$ is Schur-concave, then 
		$$\widehat C(\bar{F}_X(t),\bar{F}_Y(t))\leq \widehat C(\bar{F}_{\tilde X}(t),\bar{F}_{\tilde Y}(t)).$$
		Therefore 
		\begin{align*}
			\mathsf{GMD}(X,Y)
			&=2\int_0^\infty \left( \frac{\bar{F}_X(t)+\bar{F}_Y(t)} 2-\widehat C(\bar{F}_X(t),\bar{F}_Y(t))  \right){\rm d}t \\
			&\geq 2\int_0^\infty \left( \frac{\bar{F}_{\tilde X}(t)+\bar{F}_{\tilde Y}(t)} 2-\widehat C(\bar{F}_{\tilde X}(t),\bar{F}_{\tilde Y}(t))\right){\rm d}t \\
			&=\mathsf{GMD}(\tilde X,\tilde Y)
		\end{align*}
		and the proof is complete. The result for the bivariate Gini's index holds since $\mathsf{E}(X)+\mathsf{E}(Y)=\mathsf{E}(\tilde X)+\mathsf{E}(\tilde Y)$.
	\end{proof}
	\par
	We remark that, for $\widehat C$ Schur-concave, if $\bar{F}_X+\bar{F}_Y=\bar{F}_{\tilde X}+\bar{F}_{\tilde Y}$ and
	$X\le_{st}\tilde{X}\le_{st}\tilde{Y}\le_{st}Y$ hold, then 
	$$
	\mathsf{GMD}(X,Y)\geq\mathsf{GMD}(\tilde X,\tilde Y),\qquad\mathsf{G}(X,Y)\geq\mathsf{G}(\tilde X,\tilde Y).
	$$ 
	Thus, for a fixed copula, if one increases the dispersions between the random lifetimes involved in the vector, then the bivariate Gini's indices increase.
	\par
	The result included in Proposition \ref{prop:Schur-concave} is related with Parrondo's paradox in reliability stated in Di Crescenzo \cite{DiCrescenzo:2007} and in Navarro and Spizzichino \cite{Navarro:Spizzichino}. In this sense we know that many copulas are Schur-concave (for example all the Archimedean copulas). Hence the preceding result holds for many dependence models. If we just compare with the homogeneous case, we can relax the condition of the copula to the weakly Schur-concavity defined in Eq.\ (\ref{eq:weakly Schur}). 
	\begin{proposition}\label{prop:weakly Schur}
		Let $(X,Y)$ be a random vector with survival copula $\widehat C$. Let $(Z,\tilde{Z})$ be another random vector with the same survival copula and common marginal s.f.\ $\bar{G}$. If $\widehat C$ is weakly Schur-concave (respectively Schur-convex) and $\bar{G}= (\bar{F}_X+\bar{F}_Y)/2$, then 
		$$
		\mathsf{GMD}(X,Y)\geq\mathsf{GMD}(Z,\tilde{Z})\,(\leq),\qquad\mathsf{G}(X,Y)\geq\mathsf{G}(Z,\tilde{Z})\,(\leq).
		$$
	\end{proposition}
	\begin{proof}
		Under the stated assumptions, if $\widehat C$ is weakly Schur-concave, then  
		$$
		\widehat C(\bar{F}_X(t),\bar{F}_Y(t))\leq \widehat C(\bar{G}(t),\bar{G}(t)).
		$$
		Therefore 
		\begin{align*}
			\mathsf{GMD}(X,Y)
			&=2\int_0^\infty \left( \frac{\bar{F}_X(t)+\bar{F}_Y(t)} 2-\widehat C(\bar{F}_X(t),\bar{F}_Y(t))  \right){\rm d}t \\
			&\geq 2\int_0^\infty \left( \bar{G}(t)-\widehat C(\bar{G}(t),\bar{G}(t))\right){\rm d}t \\
			&=\mathsf{GMD}(Z,\tilde{Z})
		\end{align*}
		and the thesis follows. The proof for weakly Schur-convex copulas is similar. The result for the bivariate Gini's index holds since $\mathsf{E}(X)+\mathsf{E}(Y)=2\mathsf{E}(Z)$.
	\end{proof}
	Recalling Eq.\ (\ref{eq:Schur-constant}), in the following we consider the Schur-constant assumption for $\widehat{C}$. 
	\begin{proposition}
		Let $\widehat{C}$ be the survival copula of $(X,Y)$. If $\widehat{C}$ is Schur-constant, then
		\begin{equation}\label{eq:prop Schur-constant}
			\mathsf{GMD}(X,Y)=\mathsf{E}(X),\qquad\mathsf{G}(X,Y)=\frac12.
		\end{equation}
	\end{proposition}
	\begin{proof}
		If $\widehat{C}$ is Schur-constant, then $|X-Y|=_{st}X$, as shown in Theorem 4 of Nelsen \cite{Nelsen:2005}. Therefore, the thesis follows recalling Eqs.\ (\ref{eq:bivariateGMD}) and (\ref{eq:bivariateGindex}). 
	\end{proof}
	\par
	Clearly, from the left-hand-side of Eq.\ (\ref{eq:prop Schur-constant}), for $(X,Y)\sim\widehat{C}_1$ and $(\tilde{X},\tilde{Y})\sim\widehat{C}_2$ with Schur-constant copulas, if $X\le_{st}\tilde{X}$, then 
	$\mathsf{GMD}(X,Y)\le\mathsf{GMD}(\tilde{X},\tilde{Y})$.
	\qquad
	\\ 
	\par 
	We conclude this section by extending the bivariate Gini's indices to the multivariate case. Given an $n$-dimensional random vector $(X_1,\dots,X_n)$, we use the following notation 
	$$
	X_{1:n}=\min\{X_1,\dots,X_n\}\qquad\text{and}\qquad X_{n:n}=\max\{X_1,\dots,X_n\}.
	$$
	\begin{definition}\label{def:multivariateGMD}
		Let ${\bf X}=(X_1,\dots,X_n)$ be an $n$-dimensional random vector. The multivariate Gini's mean difference of ${\bf X}$ is defined as 
		\begin{equation}\label{eq:multivariateGMD}
			\mathsf{GMD}({\bf X})=\mathsf{GMD}(X_{1:n},X_{n:n}).
		\end{equation}
	\end{definition}
	Moreover, by denoting with $R\coloneqq X_{n:n}-X_{1:n}\ge0$, the range of the random vector ${\bf X}$, one has  $\mathsf{GMD}({\bf X})=\mathsf{E}(X_{n:n})-\mathsf{E}(X_{1:n})=\mathsf{E}(R)$. Clearly $\mathsf{GMD}({\bf X})=0$ if and only if $X_1\equiv\dots\equiv X_n$. In this sense the multivariate Gini's mean difference can be used to measure the dispersion of the random vector ${\bf X}$.
	\par
	According to Eq.\ (\ref{eq:bivariateGindex}), one can define a multivariate ratio based on Eq.\ (\ref{eq:multivariateGMD}), namely {\em multivariate Gini's index} of ${\bf X}$, as follows 
	\begin{equation*}\label{eq:multivariateGindex}
		\mathsf{G}({\bf X})\coloneqq\frac{\mathsf{GMD}({\bf X})}{\mathsf{E}(X_{1:n})+\mathsf{E}(X_{n:n})}\in[0,1].
	\end{equation*}
	\par
	If $X_1,\dots,X_n$ are i.d., with common s.f.\ $\bar{F}$, having copula $C$ and survival copula $\widehat{C}$, Eq.\ (\ref{eq:multivariateGMD}) becomes
	\begin{equation}\label{eq:multivariateGMD i.d.}
		\mathsf{GMD}({\bf X})=\int_{0}^{+\infty}\left\{1-\delta\left(F(t)\right)-{\hat{\delta}}\left(\bar{F}(t)\right)\right\}{\rm d}t,
	\end{equation}
	where $\delta$ and $\hat{\delta}$ are the diagonal sections of $C$ and $\widehat{C}$, respectively. Moreover, if $X_1,\dots,X_n$ are i.i.d., having common s.f.\ $\bar{F}$, then Eq.\ (\ref{eq:multivariateGMD i.d.}) reduces to
	\begin{equation}\label{eq:multivariateGMD i.i.d.}
		\mathsf{GMD}({\bf X})=\int_{0}^{+\infty}\left\{1-\left(F(t)\right)^n-\left(\bar{F}(t)\right)^n\right\}{\rm d}t.
	\end{equation}
	\begin{remark}
		If $X_i\sim\mathcal{U}(0,1)$, for $i\in\{1,\dots,n\}$, and they are independent, then Eq.\ (\ref{eq:multivariateGMD i.i.d.}) becomes 
		$$
		\mathsf{GMD}({\bf X})=\int_{0}^{1}\left\{1-t^n-(1-t)^n\right\}{\rm d}t=\frac{n-1}{n+1}.
		$$
		Hence, $\mathsf{GMD}({\bf X})\rightarrow 1$ as $n\rightarrow+\infty$. For instance, when $X_1,\dots,X_n$ describe the lifetimes of $n$ components in a system, then $\mathsf{GMD}({\bf X})$ goes to $1$ when the number of the components increases significantly. In addition, since in this case $\mathsf{E}(X_{n:n})=n/(n+1)$ and $\mathsf{E}(X_{1:n})=1/(n+1)$, it follows that $\mathsf{GMD}({\bf X})=\mathsf{G}({\bf X})$. As examples, for $n=2,3,4,5$ one has $\mathsf{G}({\bf X})=\frac13,\frac12,\frac35,\frac23$, respectively. 
	\end{remark} 
	\begin{remark}
		If $X_i\sim{\rm Exp}(1)$, for $i\in\{1,\dots,n\}$, and they are independent, by setting $t=e^{-x}$, Eq.\ (\ref{eq:multivariateGMD i.i.d.}) becomes 
		$$
		\mathsf{GMD}({\bf X})=\int_{0}^{1}\frac{\left\{1-(1-t)^n-t^n\right\}}{t}{\rm d}t=H(n)-\frac{1}{n},
		$$
		where $H(n)=\sum_{k=1}^{n}\frac1k$ is the truncation of the harmonic series. 
		One has $\mathsf{GMD}({\bf X})\rightarrow+\infty$ as $n\rightarrow+\infty$.
		In addition, it follows 
		$$
		\mathsf{G}({\bf X})=\frac{nH(n)-1}{nH(n)+1},
		$$
		that goes to $1$ as $n$ increases. As examples, for $n=2,3,4,5$ one has $\mathsf{G}({\bf X})=\frac12,\frac9{13},\frac{11}{14},\frac{125}{149}$, respectively. 
	\end{remark}

	\section{Bounds and inequalities}\label{sec:bounds and inequalities}
	This section is devoted to show bounds and inequalities for the bivariate Gini's indices defined in Section \ref{sec:bivariateGMD}. Bounds for the univariate Gini's indices given in Eqs.\ (\ref{eq:GMD}) and (\ref{eq:Gindex}) can be found, for instance, in Yin et al.\ \cite{Yin:etal}. In the next result we use the Jensen's inequality in Eq.\ (\ref{eq:Jensen's inequality}) aiming to provide lower bounds for Eqs.\ (\ref{eq:bivariateGMD}) and (\ref{eq:bivariateGindex}).
	\begin{proposition}
		Let $(X,Y)$ be a random vector. Then 
		\begin{equation}\label{eq:Jensen's bound}
			\mathsf{GMD}(X,Y)\ge|\mathsf{E}(X)-\mathsf{E}(Y)|,\qquad	\mathsf{G}(X,Y)\ge\frac{|\mathsf{E}(X)-\mathsf{E}(Y)|}{\mathsf{E}(X)+\mathsf{E}(Y)}.
		\end{equation}
	\end{proposition} 
	\begin{proof}
		The thesis immediately follows by recalling Eqs.\ (\ref{eq:bivariateGMD}) and (\ref{eq:bivariateGindex}) and making use of Eq.\ (\ref{eq:Jensen's inequality}) with $\psi$ equal to the absolute value function and $Z=X-Y$.
	\end{proof}
	\par
	Clearly, if $\mathsf{E}(X)\le\mathsf{E}(Y)$, then from Eq.\ (\ref{eq:Jensen's bound}) 
	\begin{equation}\label{eq:bound mean differences}
		\mathsf{GMD}(X,Y)\ge\mathsf{E}(Y)-\mathsf{E}(X),\qquad	\mathsf{G}(X,Y)\ge\frac{\mathsf{E}(Y)-\mathsf{E}(X)}{\mathsf{E}(Y)+\mathsf{E}(X)}.
	\end{equation}
	\begin{remark}
		Let $X$	and $Y$ be exponentially distributed having s.f.'s $\bar{F}_X(t)=e^{-\lambda_1t}$ and $\bar{F}_Y(t)=e^{-\lambda_2t}$, for $t\ge 0$, respectively. If $\lambda_1>\lambda_2>0$, from Eq.\ (\ref{eq:bound mean differences}) one has
		$$
		\mathsf{GMD}(X,Y)\ge\frac{\lambda_1-\lambda_2}{\lambda_1\lambda_2},\qquad\mathsf{G}(X,Y)\ge\frac{\lambda_1-\lambda_2}{\lambda_1+\lambda_2}.
		$$
	\end{remark}
	\par	
	Making use of the Markov's inequality in Eq.\ (\ref{eq:Markov's inequality}), below we interpret $\mathsf{GMD}(X,Y)$ as a bound for the probability that $Y$ assumes values in a neighborhood of $X$. 
	\begin{proposition}\label{prop:Markov's bound}
		Let $(X,Y)$ be a random vector. For $a>0$, one has 
		\begin{equation}\label{eq:Markov's bound}
			\mathsf{Pr}\left(X-a<Y<X+a\right)\ge 1-\frac{\mathsf{GMD}(X,Y)}{a}.
		\end{equation}
	\end{proposition} 
	\begin{proof}
		The thesis follows making use of Eq.\ (\ref{eq:Markov's inequality}) with $Z=|X-Y|$, and by recalling Eq.\ (\ref{eq:bivariateGMD}).
	\end{proof}
	\par
	In addition, one can use the Markov's inequality in Eq.\ (\ref{eq:Markov's inequality}) in order to compare the bivariate Gini's mean difference, for i.d.\ marginals, with the univariate one, as shown in the next result. 
	\begin{theorem}
		Let $(X,Y)$ be a random vector having i.d.\ marginals with mean $\mu$. One has 
		$$
		\mathsf{Pr}\left(\frac{X}{\mu}-2\mathsf{G}(X)<\frac{Y}{\mu}<\frac{X}{\mu}+2\mathsf{G}(X)\right)\ge 1-\frac{\mathsf{G}(X,Y)}{\mathsf{G}(X)}.
		$$
	\end{theorem}
	\begin{proof}
		Since $X=_{st}Y$, recalling Eqs.\ (\ref{eq:GMD}), (\ref{eq:Gindex}), (\ref{eq:bivariateGMD}) and (\ref{eq:bivariateGindex}), one has 
		$$
		\frac{\mathsf{GMD}(X,Y)}{\mathsf{GMD}(X)}=\frac{\mathsf{G}(X,Y)}{\mathsf{G}(X)}. 
		$$
		Therefore, the thesis follows from Eq.\ (\ref{eq:Markov's bound}) for $a=\mathsf{GMD}(X)$ and dividing by $\mu$ in the inequalities stated in the left-hand-side. 
	\end{proof}
	\par
	In the next result, making use of the FH bounds in Eq.\ (\ref{eq:Fréchet-Hoeffding bounds}) and Proposition \ref{prop:concordance ordering}, we obtain attainable bounds for the bivariate Gini's mean difference. The bounds for the bivariate Gini's index can be obtained dividing by $\mathsf{E}(X)+\mathsf{E}(Y)$.
	\begin{proposition}\label{prop:Fréchet-Hoeffding bounds bivariate GMD}
		Let $(X,Y)$ be a random vector with marginal s.f.'s $\bar{F}_X$ and $\bar{F}_Y$. Then
		\begin{equation}\label{eq:Fréchet-Hoeffding bounds bivariate GMD}
			\int_{0}^{+\infty}|\bar{F}_X(t)-\bar{F}_Y(t)|{\rm d}t\le\mathsf{GMD}(X,Y)\le\int_{0}^{+\infty}\left\{1-|\bar{F}_X(t)+\bar{F}_Y(t)-1|\right\}{\rm d}t,
		\end{equation}
		with bounds reached when the copula of $(X,Y)$ coincides with the FH bounds in Eq.\ (\ref{eq:Fréchet-Hoeffding bounds}).
	\end{proposition}
	\begin{proof}
		By recalling Eq.\ (\ref{eq:bivariateGMD s.f.}), from the FH upper bound one has
		$$
		\begin{aligned}
			\mathsf{GMD}(X,Y)&\ge\int_{0}^{+\infty}\left\{\bar{F}_X(t)+\bar{F}_Y(t)-2\min\{\bar{F}_X(t),\bar{F}_Y(t)\}\right\}{\rm d}t\\
			&=\int_{0}^{+\infty}|\bar{F}_X(t)-\bar{F}_Y(t)|{\rm d}t,\\
		\end{aligned}
		$$
		where the last equality follows from $|x-y|=x+y-2\min\{x,y\}$. Similarly, from the FH lower bound we set
		$$
		\begin{aligned}
			\mathsf{GMD}(X,Y)&\le\int_{0}^{+\infty}\left\{\bar{F}_X(t)+\bar{F}_Y(t)-2\max\{\bar{F}_X(t)+\bar{F}_Y(t)-1,0\}\right\}{\rm d}t\\
			&=\int_{0}^{+\infty}\left\{1-|\bar{F}_X(t)+\bar{F}_Y(t)-1|\right\}{\rm d}t,\\
		\end{aligned}
		$$
		where the last equality follows from $2\max\{x,y\}=|x-y|+x+y$. This conclude the proof.
	\end{proof}
	\par
	We remark that, if $X\le_{st}Y$, then the lower bound in Eq.\ (\ref{eq:Fréchet-Hoeffding bounds bivariate GMD}) coincides with the one given in Eq.\ (\ref{eq:bound mean differences}), attained when $C=M$.
	\begin{remark}
		Suppose that $X\sim\mathcal{U}(0,a)$ and $Y\sim\mathcal{U}(0,b)$, with $0<a<b$. For any copula $C$, from Eq.\ (\ref{eq:Fréchet-Hoeffding bounds bivariate GMD}) with some calculations it follows 
		$$
		\frac{b^2-a^2}{2(a+b)}\le\mathsf{GMD}(X,Y)\le\frac{a^2+b^2}{2(a+b)},\qquad\frac{b^2-a^2}{(a+b)^2}\le\mathsf{G}(X,Y)\le\frac{a^2+b^2}{(a+b)^2}.
		$$
		Suppose now that $X\sim{\rm Exp}(1)$ and $Y\sim\mathcal{U}(0,1)$. Then from Eq.\ (\ref{eq:Fréchet-Hoeffding bounds bivariate GMD}) with some calculations one obtains 
		$$
		0.5\le\mathsf{GMD}(X,Y)\le0.955937,\qquad0.333333\le\mathsf{G}(X,Y)\le0.637291.
		$$		
	\end{remark}
	\par
	In the next result we provide an upper bound for the bivariate Gini's mean difference when $X=_{st}Y$, with particular attention to the cases in which the median of $X$ coincides with its mean value.
	\begin{proposition}\label{prop:Fréchet-Hoeffding upper bound ID bivariate GMD}
		Let $(X,Y)$ be a random vector with i.d.\ marginals having c.d.f.\ $F$, mean $\mu$ and median $m$. Then 
		\begin{equation}\label{eq:Fréchet-Hoeffding upper bound ID bivariate GMD}
			0\le\mathsf{GMD}(X,Y)\le2\left(\mu-m\right)+4\int_{0}^{m}F(t){\rm d}t.
		\end{equation}
		In particular, if $\mu=m$, then $\mathsf{GMD}(X,Y)\le4\int_{0}^{m}F(t){\rm d}t$.
	\end{proposition}
	\begin{proof}
		Under these assumptions, from Eq.\ (\ref{eq:Fréchet-Hoeffding bounds bivariate GMD}), with few calculations one has
		$$
		\begin{aligned}
			\mathsf{GMD}(X,Y)&\le\int_{0}^{+\infty}1-|2\bar{F}(t)-1|{\rm d}t\\
			&=2\int_{0}^{m}F(t){\rm d}t+2\int_{m}^{+\infty}\bar{F}(t){\rm d}t\\
			&=2\left(\mu-m\right)+4\int_{0}^{m}F(t){\rm d}t,
		\end{aligned}
		$$
		where the last equality is obtained by adding and subtracting $2\int_{0}^{m}\bar{F}(t){\rm d}t$. 
	\end{proof}
	\begin{remark}\label{rem:bound}
		If $X$ and $Y$ are uniformly distributed over $[0,b]$, with $b>0$, from Eq.\ (\ref{eq:Fréchet-Hoeffding upper bound ID bivariate GMD}), recalling that $\mu=m=1/2$, one has
		$$
		\mathsf{GMD}(X,Y)\le\frac{b}2,\qquad\mathsf{G}(X,Y)\le\frac12.
		$$
		If $X$ and $Y$ are exponentially distributed with mean $\mu>0$, from Eq.\ (\ref{eq:Fréchet-Hoeffding upper bound ID bivariate GMD}) it follows 
		$$
		\mathsf{GMD}(X,Y)\le2\mu\ln(2),\qquad\mathsf{G}(X,Y)\le\ln(2).
		$$
	\end{remark}
	\begin{remark}
		Making use of the Jensen's inequality in Eq.\ (\ref{eq:Jensen's inequality}), it is easy to see that the median and the mean satisfy the following well known relation
		$$
		|\mu-m|\le\sigma,
		$$
		where $\sigma$ is the standard deviation. Analogously, for an absolutely continuous random lifetime $X$, from Markov's inequality in Eq.\ (\ref{eq:Markov's inequality}) we get
		$$
		{\rm Pr}(X\ge m)=\frac12\le\frac{\mu}{m}
		$$
		and therefore $\mu\ge m/2$. Moreover, from Eq.\ (\ref{eq:Fréchet-Hoeffding upper bound ID bivariate GMD}) one can obtain another lower bound for the mean as 
		$$
		\mu\ge m-2\int_{0}^{m}F(t){\rm d}t.
		$$ 
	\end{remark}
	\par
	In the next result we give bounds for the bivariate Gini's mean difference when $X$ and $Y$ are stochastically ordered, by using the bivariate Gini's mean differences of their respectively homogeneous cases. The bounds for the bivariate Gini's index follow dividing by $\mathsf{E}(X)+\mathsf{E}(Y)$.
	\begin{proposition}
		Let $(X,Y),(X,\tilde{X})$ and $(Y,\tilde{Y})$ be three random vectors with the same survival copula $\widehat{C}$, where $X=_{st}\tilde{X}$ and $Y=_{st}\tilde{Y}$, respectively. If $X\le_{st}Y$, then
		\begin{equation}\label{eq:bounds bivariate GMD}
			\mathsf{GMD}(Y,\tilde{Y})-2\left(\mathsf{E}(Y)-\mathsf{E}(X)\right)\le\mathsf{GMD}(X,Y)\le\mathsf{GMD}(X,\tilde{X})+2\left(\mathsf{E}(Y)-\mathsf{E}(X)\right).
		\end{equation}
	\end{proposition}
	\begin{proof}
		We only show how to obtain the right-hand-side of Eq.\ (\ref{eq:bounds bivariate GMD}), since the left-hand-side similarly follows. Under the stated assumptions, since $\bar{F}_X\le_{st}\bar{F}_Y$ one has $\mathsf{E}(X)\le\mathsf{E}(Y)$ and $\mathsf{E}(\min\{X,\tilde{X}\})\le\mathsf{E}(\min\{X,Y\})$, and therefore
		$$
		\mathsf{GMD}(X,Y)\le2\mathsf{E}(Y)-2\mathsf{E}(\min\{X,\tilde{X}\}).
		$$
		The proof follows making use of Eq.\ (\ref{eq:bivariateGMD mean}) in $\mathsf{E}(\min\{X,\tilde{X}\})$.
	\end{proof}
	\par
	In the independent case, the Gini's indices in the bounds given in Eq.\ (\ref{eq:bounds bivariate GMD}) can be replaced with the respectively univariate indices defined in Eq.\ (\ref{eq:GMD}).

	\section{Efficiency Gini's indices in systems}\label{sec:effGMD}
	In this section we define particular versions of the bivariate Gini's indices, for the cases in which the random lifetimes $X$ and $Y$ in Eqs.\ (\ref{eq:bivariateGMD}) and (\ref{eq:bivariateGindex}) are suitable systems. First, we recall that a (binary) system with (binary) components of order $n$ is a Boolean structure function (map) $\Phi:\{0,1\}^n\rightarrow\{0,1\}$, where $\Phi\left(x_1,\dots,x_n\right)\in\{0,1\}$ represents the system's state that is completely determined by the components' states represented by $x_1,\dots,x_n\in\{0,1\}$.
	A {\em semi-coherent system} of order $n$ is a system $\Phi$ that is increasing and such that $\Phi\left(0,\dots,0\right)=0$ and $\Phi\left(1,\dots,1\right)=1$. In addition, we say that $\Phi$ is a {\em coherent system} of order $n$ if it is increasing and strictly increasing in each variable in at least a point. We remark that this function $\Phi$ can be extended to real numbers and then the random lifetime $T$ of the system can be written as $T=\Phi(X_1,\dots,X_n)$, where $X_1,\dots,X_n$ are the random lifetimes of the components. For more details about systems see Navarro \cite{Navarro:2022}.
	\par
	We now are ready to define efficiency versions of the new Gini's indices. The purpose is to measure how good is a given system $T$. 
	\begin{definition}\label{def:effGMD}
		Let $(X_1,\dots,X_n)$ be a random vector and let $T=\Phi(X_1,\dots,X_k)$ be the lifetime of any semi-coherent system, for $k\le n$. Consider $X_{1:n}=\min\{X_1,\dots,X_n\}$. The efficiency Gini's mean difference of order $n$ of $T$ is defined as
		\begin{equation}\label{eq:effGMD}
			\mathsf{GMD}_n(T)=\mathsf{GMD}(X_{1:n},T).
		\end{equation}
	\end{definition}
	\par
	Clearly, from Eq.\ (\ref{eq:bivariateGMD}), one has $\mathsf{GMD}_n(T)=\mathsf{E}(T)-\mathsf{E}(X_{1:n})$. 
	In addition, $\mathsf{GMD}_n(X_{1:n})=0$ and, recalling Eq.\ (\ref{eq:multivariateGMD}),  $\mathsf{GMD}_n(X_{n:n})=\mathsf{GMD}({\bf X})$.
	\par
	From Eq.\ (\ref{eq:bivariateGindex}), one can obtain a ratio based on Eq.\ (\ref{eq:effGMD}) as $\mathsf{GMD}_n(T)/(\mathsf{E}(T)+\mathsf{E}(X_{1:n}))$. However, in order to give more information about the efficiency of a system, we prefer to define the {\em efficiency Gini's index} as 
	\begin{equation}\label{eq:effGindex}
		\mathsf{G}_n(T)\coloneqq\frac{\mathsf{GMD}_n(T)}{\mathsf{GMD}_n(X_{n:n})},
	\end{equation} 
	where $\mathsf{G}_n(T)=0$ if and only if $T\equiv X_{1:n}$, and $\mathsf{G}_n(T)=1$ if and only if $T\equiv X_{n:n}$. Therefore, since Eq.\ (\ref{eq:effGMD}) represents the distance between the expected first component failure and the expected failure of $T$, one can measure the efficiency of any semi-coherent system by using Eq.\ (\ref{eq:effGindex}). Hence, the efficiency of a system in terms of its duration increases as $\mathsf{G}_n(T)$ increases. For predictions of $T$ from $X_{1:n}$ see Navarro et al.\ \cite{Navarro:etal1}. 
	\par
	Making use of Eq.\ (\ref{eq:Markov's bound}) we provide below a useful interpretation of the efficiency Gini's index. 
	\begin{proposition}
		Under the assumptions of Definition \ref{def:effGMD}, for $c>0$ one has
		\begin{equation}\label{eq:Markov's bound eff}
			\mathsf{Pr}\left(T<X_{1:n}+c\,\mathsf{GMD}_n(X_{n:n})\right)\ge1-\frac{\mathsf{G}_n(T)}c. 
		\end{equation}
	\end{proposition}
	\begin{proof}
		The thesis immediately follows from Eq.\ (\ref{eq:Markov's bound}) when 
		$a=c\,\mathsf{GMD}_n(X_{n:n})$.
	\end{proof}
	\par
	Hence, Eq.\ (\ref{eq:Markov's bound eff}) leads to consider the efficiency Gini's index as a “control index" for the probability that the system fails after the first failure in an instant given by the difference between the expectations of the first and the last component failures.
	\par
	One could define a dual version of Eq.\ (\ref{eq:effGMD}) for any semi-coherent system $T$, referring to $X_{n:n}$ instead of $X_{1:n}$, as $\mathsf{GMD}(T,X_{n:n})$. In this case the corresponding Gini's index does not describe the same efficiency pointed out in Eq.\ (\ref{eq:effGindex}). Indeed, $\mathsf{GMD}(T,X_{n:n})$ measures the distance between the last component failure and the failure of the system, and therefore it represents a sort of ``inefficiency". 
	\par 
	We now provide a signature representation of the efficiency Gini's mean difference defined in Eq.\ (\ref{eq:effGMD}). Samaniego \cite{Samaniego:1985} (see also Samaniego \cite{Samaniego:2007}) introduced the first signature representation for the reliability of a coherent system. The interpretation is intuitively justified by the fact that a coherent system $T$ is going to fail with a component failure, and therefore the reliability of $T$ can viewed as a weighted sum of the ordered statistics (ordered component failures) $X_{i:n}$ for $i\in\{1,\dots,n\}$. Hence, if $T$ has i.i.d.\ component lifetimes $X_1,\dots,X_n$ with continuous s.f.\ $\bar{F}$, then 
	\begin{equation}\label{eq:s.f. signature}
		\bar{F}_T(t)=\sum_{i=1}^{n}s_i\bar{F}_{i:n}(t),
	\end{equation}
	for all $t$, where $s_1,\dots,s_n$ are nonnegative coefficients such that $\sum_{i=1}^{n}s_i=1$ and that do not depend on $\bar{F}$, and where $\bar{F}_{i:n}$ is the reliability function of $X_{i:n}$, for $i\in\{1,\dots,n\}$. The vector ${\bf s}=\left(s_1,\dots,s_n\right)$ is called the {\em signature} of the system, where $s_i=\mathsf{Pr}\left(T=X_{i:n}\right)$ for $i\in\{1,\dots,n\}$. Recalling that, for all $t$, the s.f.\ of $X_{i:n}$ is 
	$$
	\bar{F}_{i:n}(t)=\sum_{j=0}^{i-1}\binom{n}{j}\left[F(t)\right]^j\left[\bar{F}(t)\right]^{n-j},
	$$
	(see, for instance, Proposition 2.2 in \cite{Navarro:2022}, p.\ 30), from Eq.\ (\ref{eq:s.f. signature}), by interchanging the order of summations, it follows 
	\begin{equation}\label{eq:s.f. signature, order statistics}
		\bar{F}_T(t)=\sum_{i=1}^{n}S_{n-i+1}\binom{n}{i}\left[F(t)\right]^{n-i}\left[\bar{F}(t)\right]^{i},
	\end{equation}
	where $S_j=\sum_{i=j}^{n}s_i$ for $j\in\{1,\dots,n\}$, is the probability that the system works when exactly $n-j+1$ components work.
	\par 
	In Capaldo et al.\ \cite{Capaldo:etal1} the authors defined the cumulative information generating function of $X$ as
	\begin{equation}\label{eq:CIGF}
		G_X(\alpha,\beta)\coloneqq\int_{0}^{+\infty}\left[F(t)\right]^{\alpha}\left[\bar{F}(t)\right]^{\beta}\,{\rm{d}}t,
	\end{equation}
	for all $(\alpha,\beta)\in\mathbb{R}^2$ such that the right-hand-side of Eq.\ (\ref{eq:CIGF}) is finite. Recalling Eq.\ (\ref{eq:GMD}), it easy is to see that $\mathsf{GMD}(X)=2 G_X(1,1)$. See \cite{Capaldo:etal1} also for different generalizations of the univariate Gini's mean difference in Eq.\ (\ref{eq:GMD}). Eq.\ (\ref{eq:CIGF}) provides a unifying mathematical tool suitable to deal with cumulative entropies based on the s.f.\ and c.d.f., introduced and studied in Rao et al.\ \cite{Rao:etal} and Di Crescenzo and Longobardi \cite{DiCrescenzo:Longobardi}, respectively. In addition, it is able to recover both generalized and fractional cumulative entropies, see
	Di Crescenzo et al.\ \cite{DiCrescenzo:etal}, Kayal \cite{Kayal:2016}, Psarrakos and Navarro \cite{Psarrakos:Navarro}, Toomaj and Di Crescenzo \cite{Toomaj:DiCrescenzo} and Xiong et al.\ \cite{Xiong:2019} as references. Marginal versions of Eq.\ (\ref{eq:CIGF}) have been also discussed in \cite{Capaldo:etal1}. In particular, we recall the cumulative residual information generating measure of $X$, defined as
	\begin{equation}\label{eq:K_X}
		K_X(\beta)\coloneqq G_X(0,\beta),
	\end{equation}
	for all $\beta\in\mathbb{R}$ such that $K_X(\beta)$ is finite. 
	\par 
	In the next result we provide an alternative expression of Eq.\ (\ref{eq:effGMD}) for $k=n$ in terms of Eq.\ (\ref{eq:CIGF}), from the signature representation discussed above. 
	\begin{theorem}\label{th:signature CIGF}
		Let $T=\Phi(X_1,\dots,X_n)$ be any coherent system having i.i.d.\ components. Then
		$$
		\mathsf{GMD}_n(T)=\sum_{i=1}^{n-1}S_{n-i+1}\binom{n}{i}G_X(n-i,i).
		$$
	\end{theorem}
	\begin{proof}
		Making use of Eq.\ (\ref{eq:s.f. signature, order statistics}) for the calculation of $\mathsf{E}(T)$ in
		Eq.\ (\ref{eq:effGMD}), the thesis immediately follows from Eq.\ (\ref{eq:CIGF}).
	\end{proof}
	The Samaniego’s representation does not necessarily hold in the general case. Another way to compute the system reliability from a signature representation is given in Navarro et al.\ \cite{Navarro:etal}, namely minimal signature representation, making use of the concept of minimal path set representation (for further details see Section 2.3 in \cite{Navarro:2022}). Therefore, if $T$ is the lifetime of a coherent (or semi-coherent) system with exc.\ component lifetimes $X_1,\dots,X_n$, then
	\begin{equation}\label{eq:s.f. minimal signature}
		\bar{F}_T(t)=\sum_{i=1}^{n}a_i\bar{F}_{1:i}(t),
	\end{equation}
	for all $t$, where $a_1,\dots,a_n$ are some integer coefficients such that $a_1+\cdots+a_n=1$, and $\bar{F}_{1:i}(t)=\mathsf{Pr}\left(\min\{X_1,\dots,X_i\}>t\right)$ for $i\in\{1,\dots,n\}$.
	The vector ${\bf a}=\left(a_1,\dots,a_n\right)$ is called the {\em minimal signature} of $T$. Similarly to Theorem \ref{th:signature CIGF}, in the next result we provide an alternative expression of Eq.\ (\ref{eq:effGMD}) in terms of Eqs.\ (\ref{eq:K_X}) and (\ref{eq:s.f. minimal signature}). 
	\begin{theorem}\label{th:minimal signature CIGF}
		Let $T=\Phi(X_1,\dots,X_n)$ be any coherent system having i.i.d.\ components. Then
		\begin{equation*}\label{eq:effGMD minimal signature}
			\mathsf{GMD}_n(T)=\sum_{i=1}^{n-1}a_iK_X(i)+(a_n-1)K_X(n).
		\end{equation*}
	\end{theorem}
	\begin{proof}
		By using Eq.\ (\ref{eq:s.f. minimal signature}) in the calculation of $\mathsf{E}(T)$ in Eq.\ (\ref{eq:effGMD}), the thesis immediately follows from Eq.\ (\ref{eq:K_X}).
	\end{proof}
	\par 
	The Samaniego’s representation can also be extended to semi-coherent systems, making use of the structural signature ${\bf s}^{(n)}$ of order $n$. A result similar to Theorem \ref{th:signature CIGF} holds for ${\bf s}^{(n)}$. Analogously, ${\bf a}^{(n)}=(a_1,\dots,a_k,0,\dots,0)$ is called the {\em minimal signature of order n} (see \cite{Navarro:2022}, pp.\ 48-49). Along the same line of Theorem \ref{th:minimal signature CIGF}, below we provide an alternative expression of Eq.\ (\ref{eq:effGMD}) in terms of Eq.\ (\ref{eq:K_X}), by using the minimal signature of order $n$.
	\begin{theorem}
		Let $T=\Phi(X_1,\dots,X_k)$ be any semi-coherent system having i.i.d.\ components, contained in a random vector $(X_1,\dots,X_n)$. Therefore
		\begin{equation}\label{eq:effGMD minimal signature of order n i.i.d.}
			\mathsf{GMD}_n(T)=\sum_{i=1}^{k}a_iK_X(i)-K_X(n).
		\end{equation}
	\end{theorem} 
	Consider now an exc.\ random vector $(X_1,\dots,X_n)$ having marginals with common s.f.\ $\bar{F}$. Then for any semi-coherent system $T=\Phi(X_1,\dots,X_k)$, for $k\le n$, one has
	\begin{equation}\label{eq:effGMD minimal signature of order n i.d.}
		\mathsf{GMD}_n(T)=\sum_{i=1}^{k}a_i\int_{0}^{+\infty}\widehat{\delta}_i(\bar{F}(t)){\rm d}t-\int_{0}^{+\infty}\widehat{\delta}_n(\bar{F}(t)){\rm d}t,
	\end{equation}
	where $\widehat{\delta}_i(u)=\widehat{C}(u,\dots,u,1,\dots,1)$, for $u$ repeated $i$-times, is the diagonal section of the copula of $(X_1,\dots,X_i)$, for $i\in\{1,\dots,n\}$ and $u\in[0,1]$. Clearly, in the i.i.d.\ case Eq.\ (\ref{eq:effGMD minimal signature of order n i.d.}) reduces to Eq.\ (\ref{eq:effGMD minimal signature of order n i.i.d.}).
	\par 
	In Table \ref{tab:effGMD 1-4}, making use of Eq.\ (\ref{eq:effGMD minimal signature of order n i.i.d.}), we compute the efficiency Gini's index for all the coherent systems with $1$-$4$ i.i.d.\ components, uniformly or exponentially distributed. Table \ref{tab:effGMD 1-4 FGM} shows the efficiency Gini's index for all coherent systems with $1$-$4$ i.d.\ components, uniformly or exponentially distributed, computed using Eq.\ (\ref{eq:effGMD minimal signature of order n i.d.}) under the following $4$-dimensional FGM copula
	\begin{equation}\label{eq:FGM diagonal Table 2}
		\widehat{C}(u_1,u_2,u_3,u_4)=u_1u_2u_3u_4\left(1+\theta(1-u_1)(1-u_2)(1-u_3)(1-u_4)\right),\qquad u_1,u_2,u_3,u_4\in[0,1],
	\end{equation}
	with diagonal section 
	$$
	\widehat{\delta}_4(u)=u^4\left(1+\theta(1-u)^4\right),\qquad u\in[0,1],
	$$
	for $\theta=1,-1$, where $\widehat{\delta}_i(u)=u^i$, for $i=1,2,3$.
	\par 
	Finally, we remark that Eq.\ (\ref{eq:effGMD minimal signature of order n i.d.}) can also be extended to the not exc.\ case, making use of minimal path set representation for semi-coherent systems (see \cite{Navarro:2022}, p.\ 37).

	
	\begin{table}[http]
		\captionsetup{margin=0.91cm}
		\caption{\small Efficiency Gini's index for all coherent systems with $1$-$4$ i.i.d.\ components, uniformly or exponentially distributed, making use of the minimal signature representation of order $4$. The values have been calculated from Eq.\ (\ref{eq:effGMD minimal signature of order n i.i.d.}), by rounding to the third decimal place.
			\label{tab:effGMD 1-4}}
		\centering
		\scalebox{0.855}{
			\begin{tabular}{l@{\hspace{0,5cm}}l@{\hspace{0,5cm}}c@{\hspace{0,5cm}}c@{\hspace{0,5cm}}c}
				\toprule
				i & $T_i$ & ${\bf a^{(4)}}$ & $\mathsf{G}_n(T)$& $\mathsf{G}_n(T)$\\
				&  &  &  $X_j\sim\mathcal{U}(0,1)$& $X_j\sim{\rm Exp}(1)$\\
				\midrule
				\vspace{0,12cm}
				1  & $X_{1:1}=X_1$  & $(1, 0, 0, 0)$ & $0.500$ & $0.409$ \\
				\vspace{0,12cm}
				2  & $X_{1:2}=\min(X_1, X_2)$ (2-series)  & $(0, 1, 0, 0)$  & $0.222$ & $0.136$\\
				\vspace{0,12cm}
				3  & $X_{2:2}=\max(X_1, X_2)$ (2-parallel) & $(2, -1, 0, 0)$ & $0.778$ & $0.682$\\
				\vspace{0,12cm}
				4  & $X_{1:3}=\min(X_1, X_2, X_3)$ (3-series)  & $(0, 0, 1, 0)$ & $0.083$ & $0.045$\\
				\vspace{0,12cm}
				5  & $\min(X_1,\max(X_2,X_3))$  & $\left(0, 2 , -1, 0\right)$ & $0.361$ & $0.227$\\
				\vspace{0,12cm}
				6  & $X_{2:3}$ (2-out-of-3) & $(0, 3, -2, 0)$ & $0.500$ & $0.318$\\
				\vspace{0,12cm}
				7  & $\max(X_1,\min(X_2,X_3))$ & $\left(1, 1, -1, 0\right)$ & $0.639$ & $0.500$\\
				\vspace{0,12cm}
				8  & $X_{3:3}=\max(X_1, X_2, X_3)$ (3-parallel) & $(3, -3, 1, 0)$ & $0.917$ & $0.864$\\
				\vspace{0,12cm}
				9  & $X_{1:4}=\min(X_1, X_2, X_3, X_4)$ (4-series)  & $(0, 0, 0, 1)$ & $0$ &$0$\\
				\vspace{0,12cm}
				10 & $\max(\min(X_1, X_2, X_3),\min(X_2, X_3, X_4))$ & $\left(0, 0, 2, -1\right)$ & $0.167$ & $0.091$\\
				\vspace{0,12cm}
				11 & $\min(X_{2:3},X_4)$ & $\left(0, 0, 3, -2\right)$  & $0.250$ & $0.136$\\
				\vspace{0,12cm}
				12 & $\min(X_1,\max(X_2,X_3),\max(X_3,X_4))$ & $\left(0, 1, 1, -1\right)$ & $0.306$ & $0.182$\\
				\vspace{0,12cm}
				13 & $\min(X_1,\max(X_2,X_3,X_4))$ & $\left(0, 3, -3, 1\right)$  & $0.417$ & $0.273$\\
				\vspace{0,12cm}
				14 & $X_{2:4}$ (3-out-of-4) & $(0, 0, 4, -3)$ & $0.333$ & $0.182$\\
				\vspace{0,12cm}
				15 & $\max(\min(X_1, X_2), \min(X_1, X_3, X_4),
				\min(X_2, X_3, X_4))$ & $(0, 1, 2, -2)$ & $0.389$ & $0.227$\\
				\vspace{0,12cm}
				16 & $\max(\min(X_1, X_2), \min(X_3, X_4))$  & $(0, 2, 0, -1)$ & $0.444$ & $0.273$\\
				\vspace{0,12cm}
				17 & $\max(\min(X_1, X_2), \min(X_1, X_3),
				\min(X_2, X_3, X_4))$  & $(0, 2, 0, -1)$ & $0.444$ & $0.273$\\
				\vspace{0,12cm}
				18 & $\max(\min(X_1, X_2), \min(X_2, X_3),
				\min(X_3, X_4))$ & $\left(0, 3, -2, 0\right)$  & $0.500$ & $0.318$\\
				\vspace{0,12cm}
				19 & $\max(\min(X_1, \max(X_2, X_3, X_4)),
				\min(X_2, X_3, X_4))$  & $\left(0, 3, -2, 0\right)$  & $0.500$ & $0.318$\\
				\vspace{0,12cm}
				20 & $\min(\max(X_1, X_2), \max(X_1, X_3),
				\max(X_2, X_3, X_4))$ & $\left(0, 4, -4, 1\right)$  & $0.556$ & $0.364$\\
				\vspace{0,12cm}
				21 & $\min(\max(X_1, X_2), \max(X_3, X_4))$  & $\left(0, 4, -4, 1\right)$ &  $0.556$ & $0.364$\\
				\vspace{0,12cm}
				22 & $\min(\max(X_1, X_2), \max(X_1, X_3, X_4),
				\max(X_2, X_3, X_4))$ & $\left(0, 5, -6, 2\right)$  &  $0.611$ & $0.409$\\
				\vspace{0,12cm}
				23 & $X_{3:4}$ (2-out-of-4)  & $\left(0, 6, -8, 3\right)$  &  $0.667$ & $0.455$\\
				\vspace{0,12cm}
				24 & $\max(X_1, \min(X_2, X_3, X_4))$ & $\left(1, 0, 1, -1\right)$  &  $0.583$ & $0.455$\\
				\vspace{0,12cm}
				25 & $\max(X_1, \min(X_2, X_3), \min(X_3, X_4))$ & $\left(1, 2, -3, 1\right)$ & $0.694$ & $0.545$\\
				\vspace{0,12cm}
				26 & $\max(X_{2:3}, X_4)$ & $\left(1, 3, -5, 2\right)$ & $0.750$ & $0.591$\\
				\vspace{0,12cm}
				27 & $\min(\max(X_1, X_2, X_3), \max(X_2, X_3, X_4))$ & $\left(2, 0, -2, 1\right)$ & $0.833$ & $0.727$\\
				\vspace{0,12cm}
				28 & $X_{4:4} = \max(X_1, X_2, X_3, X_4)$ (4-parallel) & $(4, -6, 4, -1)$ & $1$ & $1$\\
				\bottomrule
			\end{tabular}
		}
	\end{table}
	
	\par
	\begin{table}[http]
		\captionsetup{margin=0.91cm}
		\caption{\small Efficiency Gini's index for all coherent systems considered in Table \ref{tab:effGMD 1-4}, in this case with $1$-$4$ i.d.\ components, uniformly or exponentially distributed, having FGM copula in Eq.\ (\ref{eq:FGM diagonal Table 2}), for $\theta=1,-1$ and making use of the minimal signature representation of order $4$. The values have been calculated from Eq.\ (\ref{eq:effGMD minimal signature of order n i.d.}), by rounding to the third decimal place.
			\label{tab:effGMD 1-4 FGM}}
			\centering
			\scalebox{0.825}{
			\begin{tabular}{l@{\hspace{0,5cm}}c@{\hspace{0,5cm}}c@{\hspace{0,5cm}}c@{\hspace{0,5cm}}c@{\hspace{0,5cm}}c}
				\toprule
				i & ${\bf a^{(4)}}$ & $\mathsf{G}_n(T)$ & $\mathsf{G}_n(T)$ & $\mathsf{G}_n(T)$ & $\mathsf{G}_n(T)$ \\
				&  & $X_j\sim\mathcal{U}(0,1), \theta=1$ & $X_j\sim\mathcal{U}(0,1), \theta=-1$ &  $X_j\sim{\rm Exp}(1), \theta=1$ & $X_j\sim{\rm Exp}(1), \theta=-1$ \\
				\midrule
				\vspace{0,12cm}
				1  & $(1, 0, 0, 0)$ & $0.500$ & $0.500$ &  $0.409$ & $0.409$\\
				\vspace{0,12cm}
				2  & $(0, 1, 0, 0)$ & $0.221$ & $0.224$ &  $0.135$ & $0.138$ \\
				\vspace{0,12cm}
				3  & $(2, -1, 0, 0)$ & $0.779$ & $0.776$ &  $0.683$ & $0.681$ \\
				\vspace{0,12cm}
				4  & $(0, 0, 1, 0)$ & $0.081$ & $0.086$ &  $0.044$ & $0.047$ \\
				\vspace{0,12cm}
				5  & $\left(0, 2 , -1, 0\right)$ & $0.360$ & $0.362$ &  $0.226$ & $0.228$ \\
				\vspace{0,12cm}
				6  & $(0, 3, -2, 0)$ & $0.500$ & $0.500$ & $0.317$ & $0.319$ \\
				\vspace{0,12cm}
				7  & $\left(1, 1, -1, 0\right)$ & $0.640$ & $0.638$ & $0.500$ & $0.500$ \\
				\vspace{0,12cm}
				8  & $(3, -3, 1, 0)$ & $0.919$ & $0.914$ & $0.865$ & $0.862$ \\
				\vspace{0,12cm}
				9  & $(0, 0, 0, 1)$ & $0$ & $0$ & $0$ & $0$ \\
				\vspace{0,12cm}
				10 & $\left(0, 0, 2, -1\right)$ & $0.162$ & $0.171$ & $0.087$ & $0.094$ \\
				\vspace{0,12cm}
				11 & $\left(0, 0, 3, -2\right)$ & $0.243$ & $0.257$ & $0.131$ & $0.142$ \\
				\vspace{0,12cm}
				12 & $\left(0, 1, 1, -1\right)$ & $0.302$ & $0.309$ & $0.179$ & $0.185$ \\
				\vspace{0,12cm}
				13 & $\left(0, 3, -3, 1\right)$ & $0.419$ & $0.414$ & $0.274$ & $0.272$ \\
				\vspace{0,12cm}
				14 & $(0, 0, 4, -3)$ & $0.324$ & $0.342$ & $0.175$ & $0.189$ \\
				\vspace{0,12cm}
				15 & $(0, 1, 2, -2)$ & $0.383$ & $0.395$ & $0.222$ & $0.232$ \\
				\vspace{0,12cm}
				16 & $(0, 2, 0, -1)$ & $0.441$ & $0.447$ & $0.270$ & $0.276$\\
				\vspace{0,12cm}
				17 & $(0, 2, 0, -1)$ & $0.441$ & $0.447$ & $0.270$ & $0.276$ \\
				\vspace{0,12cm}
				18 & $\left(0, 3, -2, 0\right)$ & $0.500$ & $0.500$ & $0.317$ & $0.319$ \\
				\vspace{0,12cm}
				19 & $\left(0, 3, -2, 0\right)$ & $0.500$ & $0.500$ & $0.317$ & $0.319$ \\
				\vspace{0,12cm}
				20 & $\left(0, 4, -4, 1\right)$ & $0.559$ & $0.553$ & $0.365$ & $0.362$ \\
				\vspace{0,12cm}
				21 & $\left(0, 4, -4, 1\right)$ & $0.559$ & $0.553$ & $0.365$ & $0.362$ \\
				\vspace{0,12cm}
				22 & $\left(0, 5, -6, 2\right)$ & $0.617$ & $0.605$ & $0.413$ & $0.406$ \\
				\vspace{0,12cm}
				23 & $\left(0, 6, -8, 3\right)$ & $0.676$ & $0.658$ & $0.460$ & $0.449$ \\
				\vspace{0,12cm}
				24 & $\left(1, 0, 1, -1\right)$ & $0.581$ & $0.586$ & $0.452$ & $0.457$ \\
				\vspace{0,12cm}
				25 & $\left(1, 2, -3, 1\right)$ & $0.698$ & $0.691$ & $0.548$ & $0.543$ \\
				\vspace{0,12cm}
				26 & $\left(1, 3, -5, 2\right)$ & $0.757$ & $0.743$ & $0.595$ & $0.587$ \\
				\vspace{0,12cm}
				27 & $\left(2, 0, -2, 1\right)$ & $0.838$ & $0.829$ & $0.730$ & $0.724$ \\
				\vspace{0,12cm}
				28 & $(4, -6, 4, -1)$ & $1$ & $1$ & $1$ & $1$ \\
				\bottomrule
			\end{tabular}	
		}
	\end{table}
	\par

	\section{Examples and simulations}\label{sec:examples and simulations}
	In this section we show examples of the bivariate Gini's indices defined in Section \ref{sec:bivariateGMD}. Interpretations in terms of areas also arises for the bivariate Gini's index in Eq.\ (\ref{eq:bivariateGindex}), as for the univariate one given in Eq.\ (\ref{eq:Gindex}). We also define empirical Gini's indices to approximate them from data, by providing simulations. We also compute the efficiency Gini's index for some systems. 
	\par 
	In the first example we consider bivariate Gini's indices of $(X,Y)$ with i.d.\ marginals, according to the uniform distribution, under different copulas. Similarly to the univariate case, the bivariate Gini's index turns out to be two times the area between the egalitarian line and the diagonal section of $(X,Y)$. 
	\begin{example}
		Let $(X,Y)$ be a random vector with diagonal section $\delta$. If $X$ and $Y$ are uniformly distributed over $[0,b]$, with $b>0$, from Eq.\ (\ref{eq:bivariateGMD c.d.f.}) it follows that 
		$$
		\mathsf{GMD}(X,Y)=2b\int_{0}^{1}\left\{t-\delta(t)\right\}{\rm d}t=b\left(1-2\int_{0}^{1}\delta(t){\rm d}t\right)
		$$
		and, from Eq.\ (\ref{eq:bivariateGindex}), one has 
		\begin{equation}\label{eq:bivariateGindex, ID, uniform}
			\mathsf{G}(X,Y)=2\int_{0}^{1}\left\{t-\delta(t)\right\}{\rm d}t=1-2\int_{0}^{1}\delta(t){\rm d}t,
		\end{equation}
		that does not depend on $b$. The first expression in Eq.\ (\ref{eq:bivariateGindex, ID, uniform}) is similar to the representation for the univariate Gini's index as two times the area between the Lorenz curve and the line $y=x$ (see Fig.\ 5.2 in \cite{Arnold:Sarabia}). For example, referring to Eq.\ (\ref{eq:bivariateGindex, ID, uniform}), in the left-hand-side of Fig.\ \ref{fig1} we plot the area (dark grey) in the i.i.d.\ case, i.e., for $\delta(t)=t^2$ for $t\in[0,1]$ (black line). We obtain an area of $1/6$ and so $\mathsf{G}(X,Y)=1/3$, that clearly represents also the value of the univariate Gini's index of the uniform distribution (since $X$ and $Y$ are independent). We also add the two extreme cases obtained from the FH bounds in Eq.\ (\ref{eq:Fréchet-Hoeffding bounds}). For $M$ (red line) we get the lower bound for the index $\mathsf{G}(X,Y)=0$, since $\delta(t)=t$. For $W$ (blue line) we obtain an area of $1/4$ (dark and light grey) and therefore $\mathsf{G}(X,Y)=1/2$ is the upper bound for the Gini's index of any bidimensional copula with uniform marginals (as shown in Remark \ref{rem:bound}).  
		Moreover,  referring to Eq.\ (\ref{eq:bivariateGindex, ID, uniform}), in the right-hand-side of Fig.\ \ref{fig1} one can see the areas for the diagonal section of the FGM family of copulas defined in Eq.\ (\ref{eq:FGM}) for $\theta= 0$ (black, i.i.d. case), $1,0.5$ (orange) and $-1,-0.5$ (green). 
		We provide the areas for $\theta=1,-1$  $(2/15,3/15)$ that lead to the extreme Gini's indices $\mathsf{G}(X,Y)=4/15=0.2666667$ and $\mathsf{G}(X,Y)=2/5=0.4$ for this family of copulas. Note that the case $\theta=0$ gives the bound $\mathsf{G}(X,Y)=1/3$ for positive $(\theta>0)$ and negative dependence cases $(\theta<0)$. In this family the Gini's index is decreasing with $\theta$, as expected.
		\par
		Similarly, referring to Eq.\ (\ref{eq:bivariateGindex, ID, uniform}), in the left-hand-side of Fig.\ \ref{fig2} we plot $\delta$ for the Clayton copulas defined in Eq.\ (\ref{eq:Clayton}) with $\theta=1,2,5,10,20$ (orange) and $\theta=-0.2,-0.4,-0.6,-0.8$ (green). The independent case is obtained when $\theta\to0$ (black) while $M$ is obtained when  $\theta\to\infty$ (red) and $W$ for $\theta=-1$ (blue). We provide the areas for $\theta=1$ $(0.1137056)$ and $\theta=-0.8$ $(0.2340531)$ that lead to $\mathsf{G}(X,Y)=0.2274112$ and $\mathsf{G}(X,Y)=0.4681062$, respectively. The upper bound is $0.5$ (blue line) and $1/3$ is also a bound for positive $(\theta>0)$ and negative $(\theta<0)$ dependence cases.
		Moreover, referring to Eq.\ (\ref{eq:bivariateGindex, ID, uniform}), in the right-hand-side of Fig.\ \ref{fig2} we plot $\delta$ for the Frank copulas defined in Eq.\ (\ref{eq:Frank}) with $\theta=1,2,5,10,20$ (orange) and $\theta=-1,-2,-5,-10,-20$ (green). The independent case is obtained when $\theta\to0$ (black) while $M$ is obtained when $\theta\to+\infty$ (red) and $W$ for $\theta\to-\infty$ (blue). We provide the areas for $\theta=1$ $(0.1498039)$ and $\theta=-1$ $(0.1827476)$ that lead to $\mathsf{G}(X,Y)=0.2996078$ and $\mathsf{G}(X,Y)=0.3654952$, respectively.
	\end{example}
	\begin{figure}[h!]
		\begin{center}
			\includegraphics*[scale=0.6]{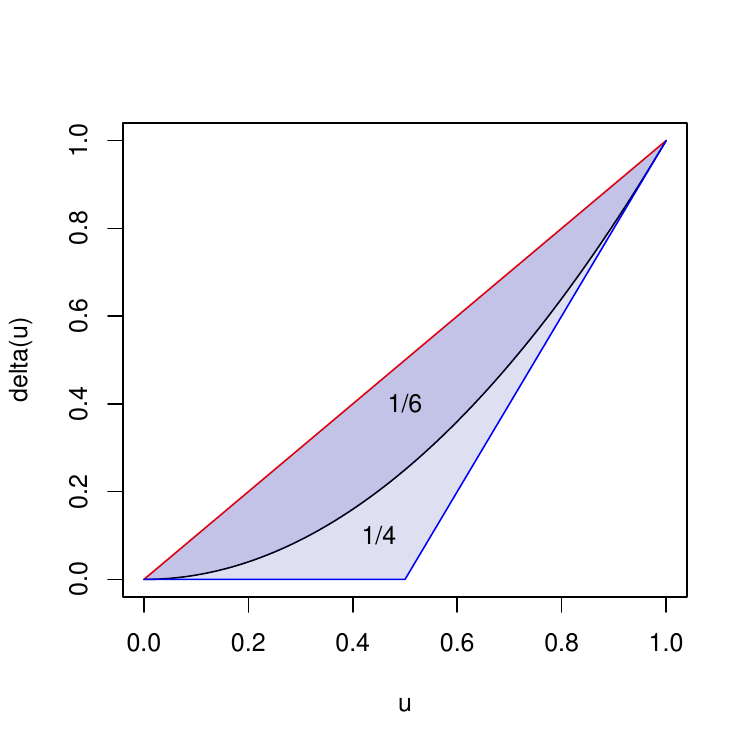}
			\includegraphics*[scale=0.6]{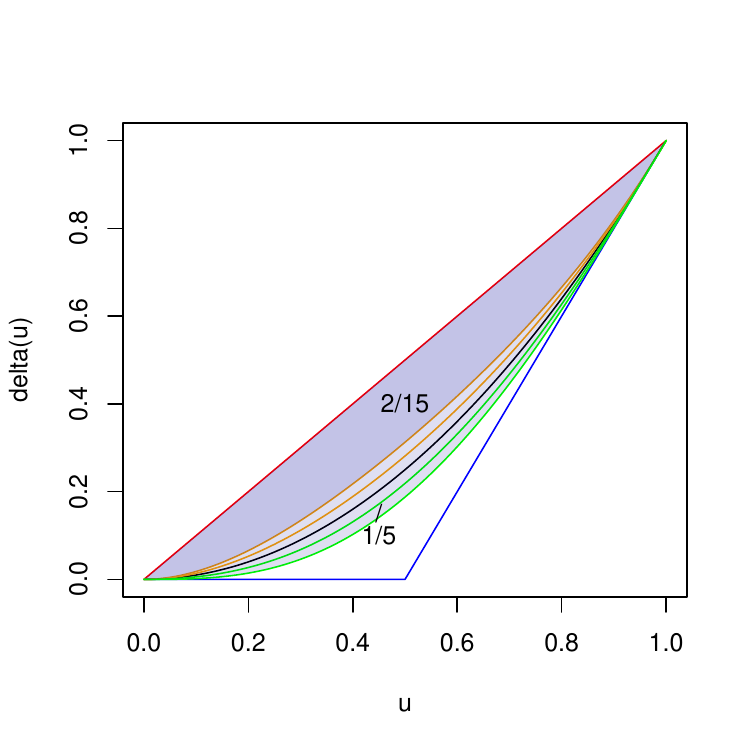}
			\captionsetup{margin=0.91cm}
			\caption{\small Plots and areas referring to Eq.\ (\ref{eq:bivariateGindex, ID, uniform}) for the diagonal section $\delta$ in the independent case (black), $M$ (red) and $W$ (blue) and FGM copulas for $\theta=1,0.5$ (orange) and $-1,-0.5$ (green).}\label{fig1}
		\end{center}
	\end{figure}
	\begin{figure}[h!]
		\begin{center}
			\includegraphics*[scale=0.6]{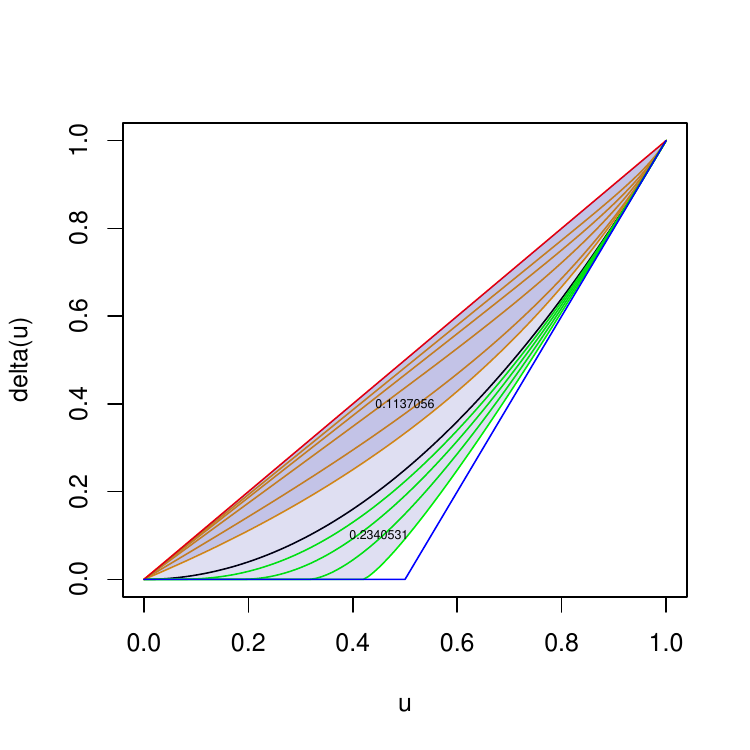}
			\includegraphics*[scale=0.6]{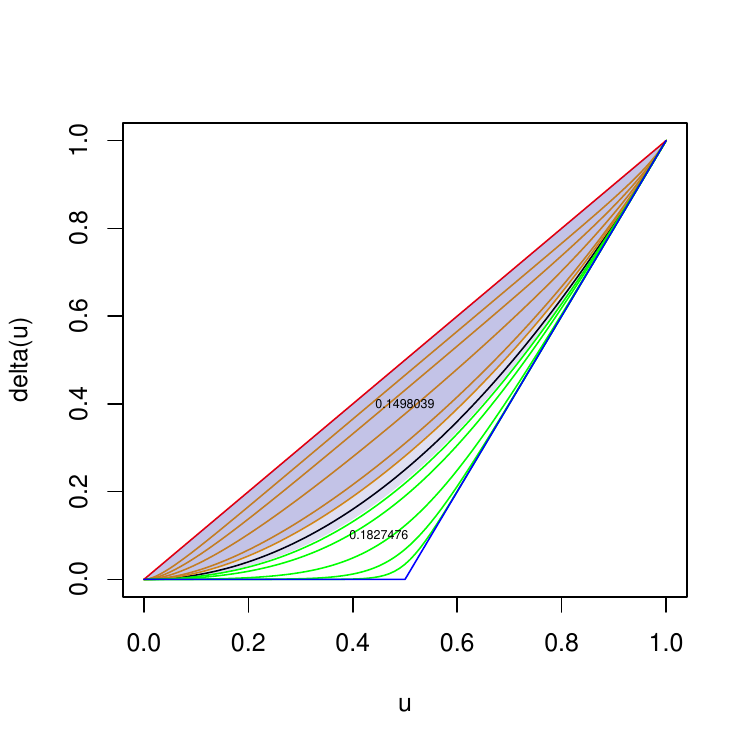}
			\captionsetup{margin=0.91cm}
			\caption{\small Plots and areas referring to Eq.\ (\ref{eq:bivariateGindex, ID, uniform}) for the diagonal section $\delta$ of Clayton copulas (left-hand-side) for $\theta=1,2,5,10,20$ (orange) and $\theta=-0.2,-0.4,-0.6,-0.8$ (green), while Frank copulas (right-hand-side) for $\theta=1,2,5,10,20$ (orange) and $\theta=-1,-2,-5,-10,-20$ (green).}\label{fig2}
		\end{center}
	\end{figure}
	\par
	In the second example we consider bivariate Gini's indices of $(X,Y)$ with i.d.\ marginals, according to the exponential distribution, under different copulas. In this model, the interpretation in terms of area is different from the univariate case. Indeed, the bivariate Gini's index turns out to be the area between the unitary line and the ratio based on the diagonal section of $(X,Y)$ and the egalitarian line. 
	\begin{example}
		Let $(X,Y)$ be a random vector with survival diagonal section $\widehat{\delta}$. 
		If $X$ and $Y$ are exponentially distributed with mean $\mu$, from Eq.\ (\ref{eq:bivariateGMD}) it follows that 
		$$
		\mathsf{GMD}(X,Y)=2\mu\int_{0}^{1}\left\{1-\frac{\hat{\delta}(t)}{t}\right\}{\rm d}t 
		$$
		and, from Eq.\ (\ref{eq:bivariateGindex}), one has 
		\begin{equation}\label{eq:bivariateGindex, ID, exponential}
			\mathsf{G}(X,Y)=1-\int_{0}^{1}\frac{\hat{\delta}(t)}{t}{\rm d}t.	
		\end{equation}
		Eq.\ (\ref{eq:bivariateGindex, ID, exponential}) just depends on the diagonal section $\widehat{\delta}(u)=\widehat C(u,u)$, for all $u\in[0,1]$, and the bivariate Gini's index can also be represented as the area between $1$ and $\hat{\delta}(u)/u$. It is easy to see that in the PQD (NQD) case one has $\widehat{\delta}(u)\geq u^2$ $(\leq)$ and $\mathsf{G}(X,Y)\leq 1/2$ $(\geq)$. For example, in the left-hand-side of Fig.\ \ref{fig3}  we plot the area (dark grey) in the i.i.d.\ case, i.e., $\hat{\delta}(t)=t^2$ for $t\in[0,1]$ (black line), and  we also add the two extreme cases obtained from the FH bounds in Eq.\ (\ref{eq:Fréchet-Hoeffding bounds}). For the i.i.d.\ case we obtain an area of $1/2$ and therefore $\mathsf{G}(X,Y)=1/2$, that is also the value of the univariate Gini's index of the exponential distribution (since $X$ and $Y$ are independent). For $M$ (red line) we get the lower bound for the index $\mathsf{G}(X,Y)=0$, while for $W$ (blue line) we obtain an area of $\ln2$ (dark and light grey) and therefore $\mathsf{G}(X,Y)=\ln2= 0.6931472$ is the upper bound for the Gini's index of any bidimensional copula with exponential marginal distributions (as shown in Remark \ref{rem:bound}). Moreover, in the right-hand-side of Fig.\ \ref{fig3} one can see the areas for the FGM family of copulas defined in Eq.\ (\ref{eq:FGM}) with $\theta=0$ (black, i.i.d. case), $1,0.5$ (orange)  and $-1,-0.5$ (green). The respective Gini's indices are $\mathsf{G}(X,Y)=0.4166667$ and $\mathsf{G}(X,Y)=0.5833333$. 
	\end{example}
	\begin{figure}[h!]
		\begin{center}
			\includegraphics*[scale=0.6]{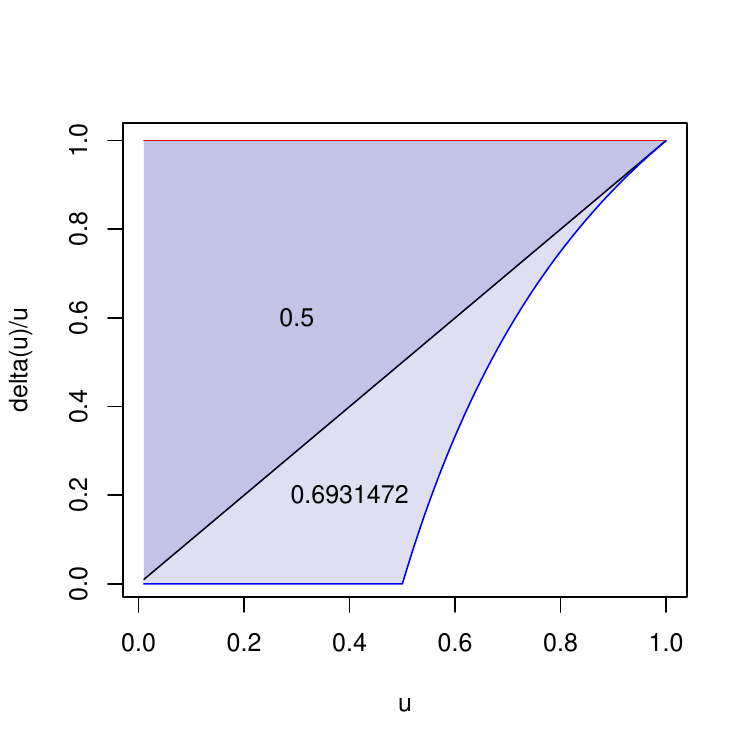}
			\includegraphics*[scale=0.6]{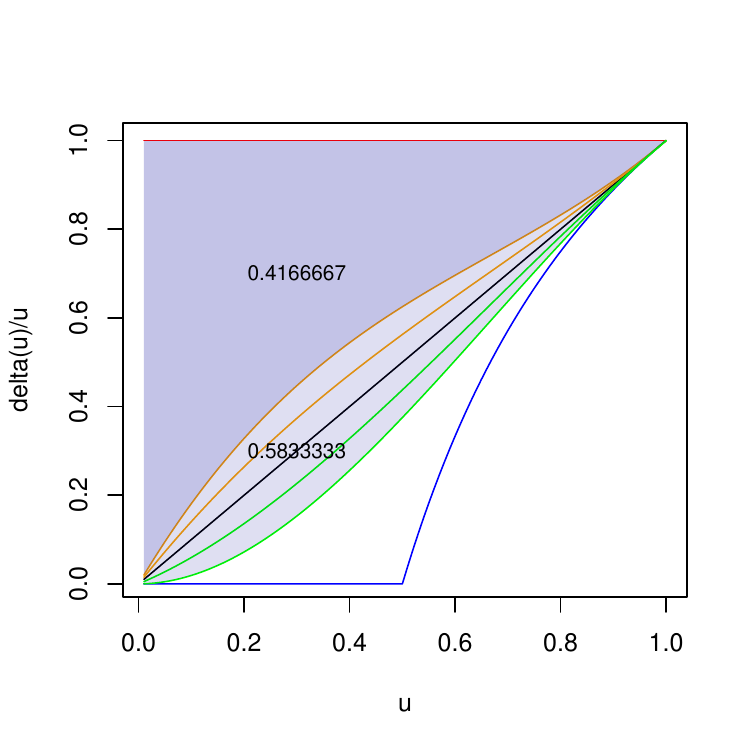}
			\captionsetup{margin=0.91cm}
			\caption{\small Plots and areas referring to Eq.\ (\ref{eq:bivariateGindex, ID, exponential}) for $\hat{\delta}(u)/u$ in the exponential independent case (black), $M$ (red) and $W$ (blue) and FGM copulas for $\theta=1,0.5$ (orange) and $-1,-0.5$ (green).}\label{fig3}
		\end{center}
	\end{figure}
	\par
	By recalling Eqs.\ (\ref{eq:mean difference}) and (\ref{eq:mean index}) we now define empirical versions of the bivariate Gini's indices introduced in Eqs.\ (\ref{eq:bivariateGMD}) and (\ref{eq:bivariateGindex}), respectively. 
	\begin{definition}\label{def:empirical bivariate Gini}
		Let $(X,Y)$ be a random vector. Let $(X_i,Y_i)$ be i.i.d.\ from $(X,Y)$, for $i\in\{1,\dots,n\}$. Consider $L_i=\min(X_i,Y_i)$ and $U_i=\max(X_i,Y_i)$. The empirical bivariate Gini's mean difference and empirical bivariate Gini's index of $(X,Y)$ are defined respectively as
		$$
		\mathsf{\widehat{GMD}}(X,Y)=\bar{U}-\bar{L},\qquad\mathsf{\widehat{G}}(X,Y)=\frac{\bar{U}-\bar{L}}{\bar{U}+\bar{L}},
		$$ 
		where $\bar{L}=\frac1n\sum_{i=1}^{n}L_i$ and $\bar{U}=\frac1n\sum_{i=1}^{n}U_i$. 
	\end{definition}
	
	Note that $Z_i=|X_i-Y_i|$, for $i\in\{1,\dots,n\}$, are i.i.d.\ random variables from $Z=|X-Y|$ and that 
	$$ 		
	\mathsf{\widehat{GMD}}(X,Y)=\frac1n\sum_{i=1}^{n}Z_i=\bar{Z},\qquad\mathsf{\widehat{G}}(X,Y)=\frac{\sum_{i=1}^{n}Z_i}{\sum_{i=1}^{n}(X_i+Y_i)}=\frac{\bar{Z}}{\bar{X}+\bar{Y}}
	$$
	and therefore we can apply here all the classic convergence theorems for the sample means. Let us see two examples. 
	\begin{example}\label{ex1}
		To get a simulated sample we consider the Clayton copula in Eq.\ (\ref{eq:Clayton}) for $\theta=1$ given by
		$$
		C(u,v)=\frac{uv}{u+v-uv},\qquad u,v\in[0,1].
		$$
		If $X=_{st}Y\sim\mathcal{U}(0,1)$, then the conditional distribution function of $Y|X=u$ for $u\in(0,1)$ is 
		$$
		C_{2|1}(v|u)=\partial_1 C(u,v)=\frac{v^2}{(u+v-uv)^2},\qquad v\in[0,1].
		$$
		In order to simulate $Y|X=u$ from a given value $u$, we obtain the inverse function of $C_{2|1}$ as  
		\begin{equation}\label{eq:inverse C}
			C^{-1}_{2|1}(z|u)=\frac{u\sqrt z}{1-(1-u)\sqrt z},\qquad z\in[0,1].
		\end{equation}
		To get a sample with $100$ data we simulate $X_i\sim\mathcal{U}(0,1)$ for $i\in\{1,\dots,100\}$ and then we obtain $Y_i$ from Eq.\ (\ref{eq:inverse C}) with $u=X_i$ and $z=Z_i\sim\mathcal{U}(0,1)$ for $i\in\{1,\dots,100\}$. Therefore we get the estimation for the bivariate Gini's index $\widehat{\mathsf{G}}(X,Y)=0.2357458$ given in Definition \ref{def:empirical bivariate Gini}.
		The exact value is $\mathsf{G}(X,Y)=2\cdot 0.1137056=0.2274112$ (see Fig.\ \ref{fig2}, left). The data obtained for $(L_i,U_i)$ can be seen in the left-hand-side of Fig.\ \ref{fig4}.
		To get a sample from standard exponential distributions if $\tilde{X}=_{st}\tilde{Y}\sim\rm{Exp}(1)$ with this survival copula we just apply the inverse transform $\bar{F}^{-1}(z)=-\ln(z)$ to the above uniform data obtaining the estimation $\widehat{\mathsf{G}}(X,Y)=0.5\cdot0.6366323=0.3183161$. The exact value is $\mathsf{G}(X,Y)=0.3068528$.
		The data obtained for $(\tilde{L}_i,\tilde{U}_i)$ can be seen in the right-hand-side of Fig.\ \ref{fig4}.
	\end{example}
	\begin{figure}[h!]
		\begin{center}
			\includegraphics*[scale=0.6]{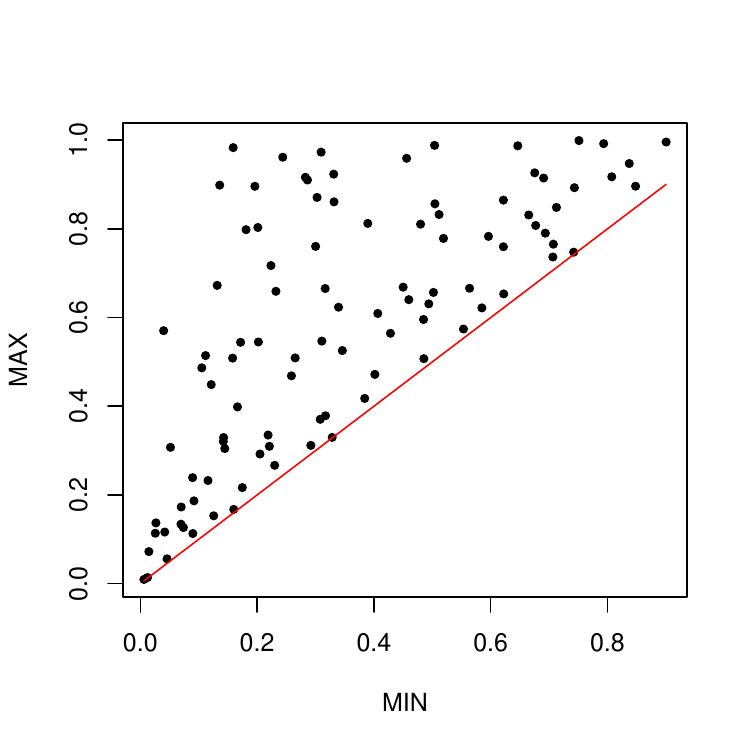}
			\includegraphics*[scale=0.6]{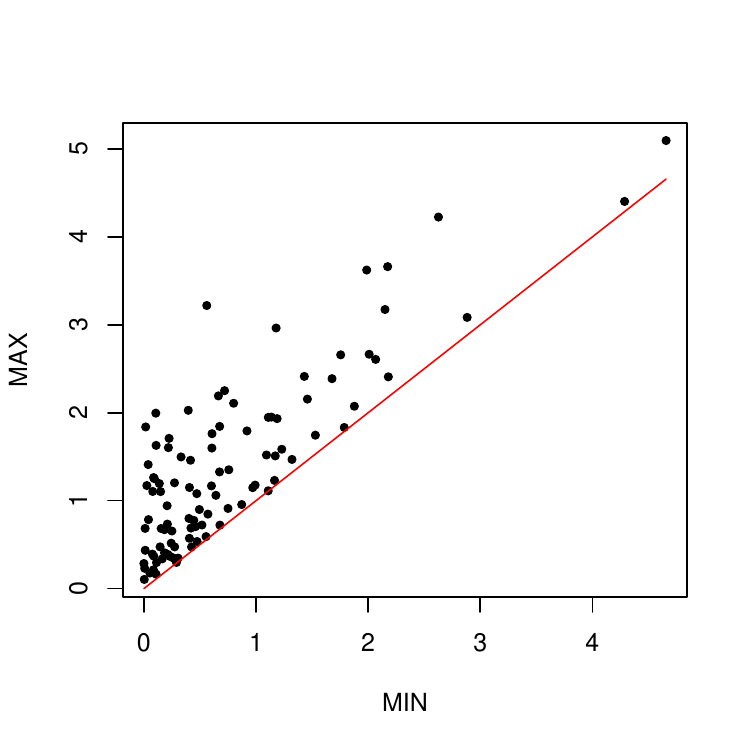}
			\captionsetup{margin=0.91cm}
			\caption{\small Simulated ordered data obtained from the Clayton copula in Example \ref{ex1} with standard uniform (left-hand-side) or exponential (right-hand-side) distributions.}\label{fig4}
		\end{center}
	\end{figure}
	\begin{example}\label{ex2}
		We now consider the Frank copula in Eq.\ (\ref{eq:Frank}) for $\theta=-1$ given by
		$$
		C(u,v)=\ln\left(1+\frac{(e^{u}-1)(e^{v}-1)}{e-1}\right),\qquad u,v\in[0,1].
		$$
		If $X=_{st}Y\sim\mathcal{U}(0,1)$, then the conditional distribution function of $Y|X=u$ for $u\in(0,1)$ is 
		$$
		C_{2|1}(v|u)=\frac{e^u(e^v-1)}{e-e^u+e^v(e^u-1)},\qquad v\in[0,1]
		$$
		and thus the inverse function of $C_{2|1}$ is  
		\begin{equation}\label{eq:inverse C 2}
			C^{-1}_{2|1}(z|u)=\ln\left[\frac{(e^u-e)z-e^u}{(e^u-1)z-e^u}\right],\qquad z\in[0,1].
		\end{equation}
		As before, from a sample with $100$ data we get the estimation for the bivariate Gini's index $\widehat{\mathsf{G}}(X,Y)=0.377477$ given in Definition \ref{def:empirical bivariate Gini} by using Eq.\ (\ref{eq:inverse C 2}). The exact value is $\mathsf{G}(X,Y)=2\cdot0.1827476=0.3654952$ (see Fig.\ \ref{fig2}, right). The data obtained for $(L_i,U_i)$ can be seen in the left-hand-side of Fig.\ \ref{fig5}.
		To get a sample from standard exponential distributions if $\tilde{X}=_{st}\tilde{Y}\sim\rm{Exp}(1)$  with this survival copula we just apply the inverse transform $\bar{F}^{-1}(z)=-\ln(z)$ to the above uniform data obtaining the estimation $\widehat{\mathsf{G}}(X,Y)=0.5\cdot1.069694=0.534847$. The exact value is $\mathsf{G}(X,Y)=0.539814$.
		The data obtained for $(\tilde{L}_i,\tilde{U}_i)$ can be seen in the right-hand-side of Fig.\ \ref{fig5}.
	\end{example}
	\par  
	In the last example we compute the empirical efficiency Gini's index of two coherent systems.
	\begin{example}\label{ex3}
		Consider the system $T_{16}$ given in Table \ref{tab:effGMD 1-4}, with i.i.d.\ components having standard exponential distributions. From a sample with $100$ data we get the estimation of the efficiency Gini's index $\widehat{\mathsf{G}}_4(T)=0.250$. The exact value is $\mathsf{G}_4(T)=0.273$ as shown in Table \ref{tab:effGMD 1-4}. The data obtained for $(X_{1:4},T)$ can be seen in the left-hand-side of Fig.\ \ref{fig6}. 
		Similarly, if one consider the system $T_{22}$ given in Table \ref{tab:effGMD 1-4}, with i.i.d.\ components having standard exponential distribution, then $\widehat{\mathsf{G}}_4(T)=0.416$. The exact value is $\mathsf{G}_4(T)=0.409$ as shown in Table \ref{tab:effGMD 1-4}. The data obtained for $(X_{1:4},T)$ can be seen in the right-hand-side of Fig.\ \ref{fig6}. 
	\end{example}
	\begin{figure}[h!]
		\begin{center}
			\includegraphics*[scale=0.6]{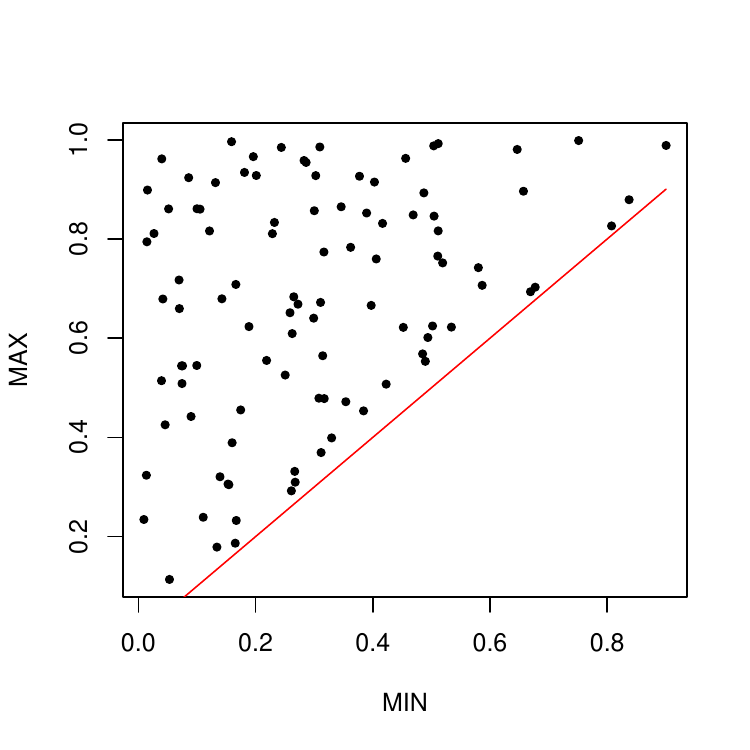}
			\includegraphics*[scale=0.6]{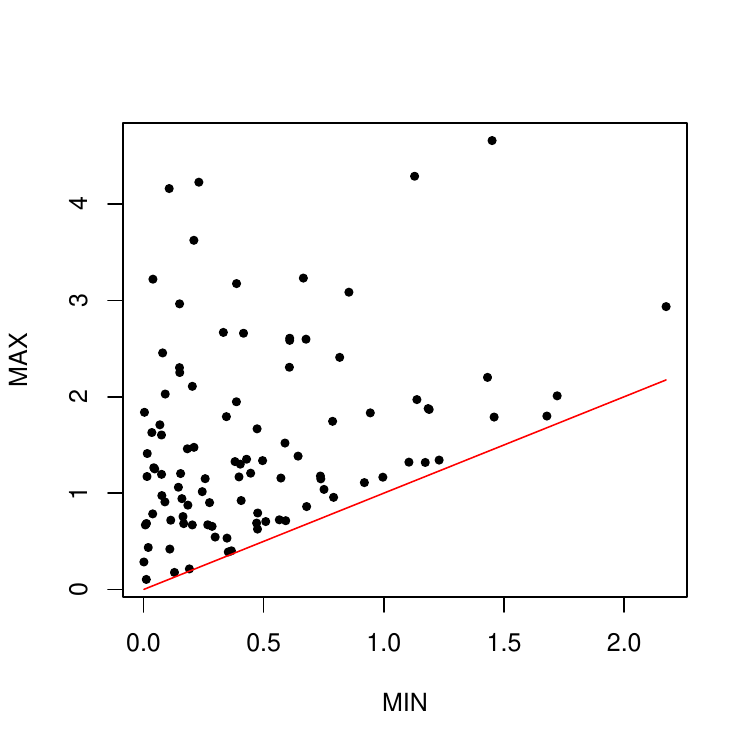}
			\captionsetup{margin=0.91cm}
			\caption{\small Simulated ordered data obtained from the Frank copula in Example \ref{ex2} with standard uniform (left-hand-side) or exponential (right-hand-side) distributions.}\label{fig5}
		\end{center}
	\end{figure}
	\begin{figure}[h!]
		\begin{center}
			\includegraphics*[scale=0.6]{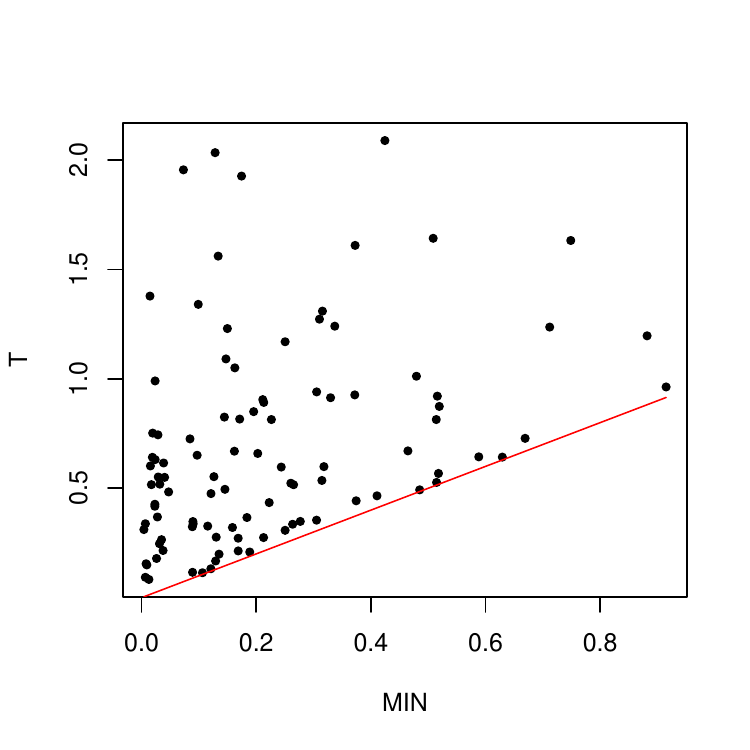}
			\includegraphics*[scale=0.6]{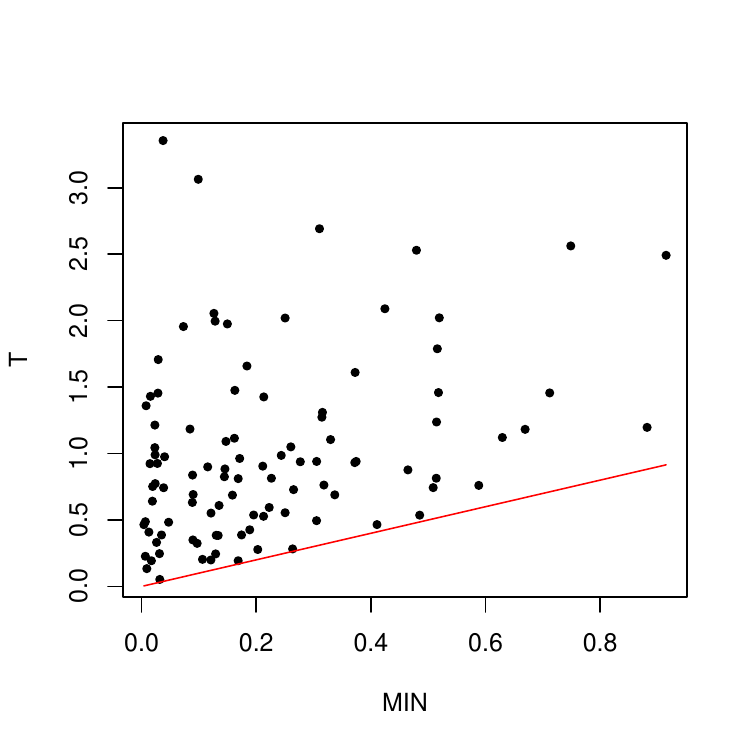}
			\captionsetup{margin=0.91cm}
			\caption{\small Simulated ordered data obtained of $(X_{1:4},T)$, for the systems $T_{16}$ (left-hand-side) and $T_{22}$ (right-hand-side) introduced in Table \ref{tab:effGMD 1-4}, with i.i.d.\ components having standard exponential distributions.}\label{fig6}
		\end{center}
	\end{figure}
	%

	\section*{Declaration of Competing Interest} 
	{\small 
		The authors declare that they have no known competing financial interests or personal relationships that could have appeared to influence the work reported in this paper.
	}
	
	\section*{Acknowledgements} 
	{\small
		M.C. is member of the group GNCS of INdAM (Istituto Nazionale di Alta Matematica). 
		This work is partially supported by MUR-PRIN 2022, project 2022XZSAFN ‘‘Anomalous Phenomena on Regular and Irregular Domains: Approximating Complexity for the Applied Sciences’’, and MUR-PRIN 2022 PNRR, project P2022XSF5H ‘‘Stochastic Models in Biomathematics and Applications’’. 
		M.C. expresses his warmest thanks to the Departamento de Estadística e Investigación Operativa of Universidad de Murcia for the hospitality during a three-month visit carried out in 2023. J.N. thanks the support of Ministerio de Ciencia e Innovación of Spain under grants PID2019-103971GB-I00/AEI/10.13039/501100011033 and MCIN/AEI/10.13039/501100011033 and the project TED2021-129813A-I00 with the support of the European Union “NextGenerationEU”/PRTR.
	}

\end{document}